\documentclass[leqno,10pt]{amsart}
\usepackage{amsmath,amsthm,amsbsy,amsfonts,amssymb}
\usepackage{MnSymbol}
\usepackage[backref=page]{hyperref}

\hypersetup{colorlinks=true,citecolor=blue,linkcolor=blue,urlcolor=blue,pdfstartview=FitH,
pdfauthor=Benoit Louvel, pdftitle=On the distribution of cubic exponential sums}

\theoremstyle{plain}
\numberwithin{equation}{section}
\newtheorem{thm}{Theorem}[section]
\newtheorem{lemme}[thm]{Lemma}
\newtheorem{prop}[thm]{Proposition}
\newtheorem{cor}[thm]{Corollary}

\theoremstyle{definition}
\newtheorem*{ack}{Acknowledgements}
\newtheorem*{notations}{Notations}
\newtheorem{defi}[thm]{Definition}

\theoremstyle{remark}
\newtheorem{rque}[thm]{Remark}

\begin{document}

\newcommand{\rond}{\mathcal}
\newcommand{\id}{ \mathfrak}
\newcommand{\einb}{\hookrightarrow} 
\newcommand{\too}{\longrightarrow} 
\newcommand{\pt}{\cdot} 
\newcommand{\pts}{\ldots}
\newcommand{\N}{\mathbb{N}} 
\newcommand{\Z}{\mathbb{Z}}
\newcommand{\Q}{\mathbb{Q}}
\newcommand{\R}{\mathbb{R}}
\newcommand{\A}{\mathbb{A}}
\newcommand{\F}{\mathbb{F}}
\newcommand{\C}{\mathbb{C}}
\newcommand{\gras}{\textbf}
\newcommand{\dual}{\land}
\newcommand{\ec}{\textrm}
\newcommand{\tq}{\vert}
\newcommand{\iso}{\cong}
\newcommand{\tild}{\widetilde}
\newcommand{\congru}{\equiv}
\newcommand{\barre}{\overline }
\newcommand{\norme}{\Vert}

\title{On the distribution of cubic exponential sums}

\author{Beno\^it Louvel}
\address{Mathematisches Institut G\"ottingen \\ Bunsenstrasse 3-5 \\ 37073 G\"ottingen \\ Germany}
\email{blouvel@uni-math.gwdg.de}
\date{\today}
\thanks{Research supported by the Graduiertenkolleg Gruppen und Geometrie 535, EPFL Lausanne and Universit\"at G\"ottingen}

\begin{abstract}
Using the theory of metaplectic forms, we study the asymptotic behavior of cubic exponential sums over the ring of Eisenstein integers. In the first part of the paper, some non-trivial estimates on average over arithmetic progressions are obtained. In the second part of the paper, we prove that the sign of cubic exponential sums changes infinitely often, as the modulus runs over almost prime integers.
\end{abstract}
\maketitle
\setcounter{page}{1}

\section{Introduction}\label{sec:intro}

\subsection{Statement of the result}

Let $f=P/Q$ be a rational function, normalized such that $P$ and $Q$ are two polynomial in $\Z[X]$ mutually coprime and with coefficients mutually coprime. Let $S_\chi(f,c)$ be the exponential sum 

\begin{equation*}
S_\chi(f,c) = \sum_{\substack{x\pmod c\\ Q(x) \barre{Q(x)} \congru 1\pmod c}} \chi(x) \exp \left(2\pi i \frac{P(x) \barre{Q(x)}}{c}\right),
\end{equation*}
where $c\in \Z-\{0\}$ and $\chi$ is a character modulo $c$. These sums satisfy the individual Weil upper bound $\tq S_\chi(f,p) \tq \leqslant k_f \sqrt{p}$, valid for almost all primes $p$, where the constant $k_f$ can be explicitly given in terms of $P$ and $Q$ (see formula (3.5.2) p.191 of \cite{del:sommes-trigo}). It is therefore natural to study the distribution of the normalized sums over the primes, i.e. to ask whether the set $\{S_\chi(f,p)/k_f \sqrt{p}\,:\, p \ec{ prime}\}$ should be expected to be equidistributed for some measure.
This is a difficult problem, and there exist actually only very few non-trivial examples for which the question of equidistribution over prime moduli has been completely solved: the case of cubic Gau\ss{} sums -- this corresponds to the choice where $P(x)=x$, $Q(x)=1$ and $\chi$ is the cubic Legendre symbol -- has been solved in  \cite{hbr-pat:kummer}, and the case of Sali\'e sums -- this corresponds to the choice where $P(x)=x^2-1$, $Q(x)=x$ and $\chi$ is the quadratic Legendre symbol -- has been solved in \cite{dfi7} (see also \cite{hoo:distrib}, for the distribution over the integers).

In this paper we consider the case of a cubic polynomial. More specifically, we shall consider the polynomial $f(x)=x^3-3x$, which is a typical situation for the problem that concerns us. Following the way initiated by Heath-Brown and Patterson in \cite{hbr-pat:kummer} and pursued latter by Livn\'e and Patterson in \cite{liv-pat}, we work over the ring of Eisenstein integers $R=\Z[\omega]$, where $\omega^3=1$, instead of working over $\Z$; the reason for that comes from the introduction of the cubic Legendre symbol (see Section\,\ref{sec:Kloo}). It will be convenient to define $e(z)$ for $z\in\C$ as $e(z)=\exp(2i\pi(z+\barre{z}))$. Let $k$ be the field $k=\Q(\omega)$ and $\rond{N}=\ec{Norm}_{k/\Q}$ the norm. In this paper, the letter $\pi$ will either denote a prime in $R$, or the value $\pi=3.14...$. The exponential sums we are interested in are defined as

\begin{equation}\label{intro:eq:0}
S(a,c) = \sum_{x\pmod c} e\left(\frac{a(x^3-3x)}{c}\right), \qquad c\in R.
\end{equation}  
These sums are real and, for almost all primes $\pi$ of $R$, they satisfy the individual Weil bound $\tq S(a,\pi)\tq \leqslant 2 \sqrt{\rond{N}(\pi)}$. The angles $\theta_{a,\pi}$ are consequently defined by $\cos \theta_{a,\pi} =S(a,\pi)/2\sqrt{\rond{N}(\pi)}$. 
The horizontal Sato-Tate law predicts that there exists a measure $\mu$ on $[0,\pi]$ such that the angles $\theta_{a,\pi}$ are equidistributed with respect to $\mu$; this means that for any $a\in R-\{0\}$ and any interval $[\alpha,\beta]$ of $[0,\pi]$, 

\begin{equation}\label{intro:eq:1}
\frac{\sharp\{\rond{N}(\pi)\leqslant X\, :\, \theta_{a,\pi}\in [\alpha,\beta]\}}{\sharp\{\rond{N}(\pi)\leqslant X\}} \overset{X\to \infty}{\longrightarrow} \mu([\alpha,\beta]).
\end{equation}
Such a conjecture seems unreachable at the moment; actually, it is even not known whether the sums $S(1,\pi)$ are positive or negative infinitely often. 

Our goal in this paper is to realize a step towards the Sato-Tate conjecture by showing that the sign of $S(1,c)$ changes infinitely often, as $c$ runs through almost prime integers.

\begin{thm}\label{intro:thm:1} Let $X\gg 1$ . There exists an explicitly computable constant $0<u<60$,  such that 

\begin{align*}
&\sharp \{ X\leqslant \rond{N}(c) < 2X\, :\, \pi\tq c \Rightarrow \rond{N}(\pi) \geqslant X^{1/u}, \ S(1,c)>0\} \gg \frac{X}{\log X}\\
&\sharp \{ X \leqslant \rond{N}(c) <2X\, :\, \pi\tq c \Rightarrow \rond{N}(\pi) \geqslant X^{1/u}, \ S(1,c)<0\} \gg \frac{X}{\log X}.
\end{align*}
\end{thm}

This result answers a question raised by Fouvry and Michel in \cite{fou-mic:chgmt-signe} (p. 9), where the authors prove an analogous result to Theorem\,\ref{intro:thm:1} for Kloosterman sums. We recall that the Kloosterman sums are defined for $m,n,c\in\Z$ by 

\begin{equation*}
K(m,n,c)=\sum_{\substack{x\pmod c \\ x \barre{x} \congru 1 \pmod c}}\exp\left(2i\pi \frac{mx+n\barre{x}}{c}\right).
\end{equation*}
In the case of Kloosterman sums, the analogue of Theorem\,\ref{intro:thm:1} was proved by Fouvry and Michel (\cite[Th\'eor\`eme\,1.2 and Th\'eor\`eme\,1.3]{fou-mic:chgmt-signe}), and the value of the corresponding $u$ is $23.9$. In the present paper, we do not try to optimize the value of the constant $u$ of Theorem\,\ref{intro:thm:1}, and content ourselves with the easily  improvable but explicit value $u\leqslant 60$. 

One of the main arguments in favor of the horizontal Sato-Tate conjecture is the vertical Sato-Tate law, concerning the distribution of the set $\{\theta_{a,\pi}\,:\, a\pmod {\pi}\}$, as the prime $\pi$ varies. 

\begin{thm}\label{intro:thm:vert}
For every interval $[\alpha,\beta] \in [0,\pi]$, 

\begin{equation*}
\frac{\sharp\{ a\pmod \pi\,:\, \theta_{a,\pi}\in [\alpha,\beta]\}}{\rond{N}(\pi)-1} \too \frac{2}{\pi} \int_\alpha^\beta \sin^2t \,dt, \qquad \ec{as } \rond{N}(\pi)\too \infty.
\end{equation*}
\end{thm}

\begin{proof}
This theorem is proved in \cite[Theorem\,7.10.5 and \S\,7.10.6]{kat:livre}. The measure $\mu$ is described as the image through $A \mapsto \arccos \left(\ec{Tr}(A)/2\right)$ of the Haar measure on $US_p(2)$. Since $US_p(2)=SU(2)$, we obtain $\mu =2/\pi \sin^2$. 
\end{proof}

\subsection{Outline of the proof of \texorpdfstring{Theorem\,\ref{intro:thm:1}}{}}

Experience has shown (see \cite{dfi7}, \cite{fou-mic:chgmt-signe} and \cite{hbr-pat:kummer}) that, in order to tackle equidistribution problems for prime moduli, it is reasonable to firstly study exponential sums over the integers and then to apply sieve techniques. The main part of this paper deals with the first part of this program, i.e. with the study of the smooth asymptotic behavior of cubic exponential sums (Theorem\,\ref{intro:thm:2}). From now on, $S(1,c)$ will be the quantity defined in \eqref{intro:eq:0}. There is a major difference between the zeta function associated to cubic exponential sums, i.e. 

\begin{equation}\label{intro:eq:2}
Z_d(s) =\sum_{c\congru 0 \pmod d} \frac{S(1,c)}{\rond{N}(c)^s},
\end{equation}
and the Kloosterman zeta function, i.e. the zeta function associated to Kloosterman sums. Actually, $Z_d(s)$ is absolutely convergent in $\Re(s)>3/2$, and has a meromorphic continuation to $\Re(s)>1$ with a pole at $s=4/3$ and possibly other poles at the spectral parameters of the hyperbolic Laplacian. This is a new feature when compared to the Kloosterman zeta function, where no poles are expected in the half-plane $\Re(s)>1$. This pole at $s=4/3$ is related to the residual spectrum of the Laplacian, which is, in turn, built on cubic theta functions (see Section\,\ref{sec:meta}). Throughout the text, we will consequently refer to the residue of the zeta function at this pole as the theta-term. 
Selberg's conjecture on the spectrum of hyperbolic surfaces predicts that $4/3$ is the only pole of $Z_d(s)$. This would imply that for any $\varepsilon > 0$,

\begin{equation}\label{intro:eq:3}
\sum_{\substack{c\congru 0 \pmod d \\ \rond{N}(c)\leqslant X}} \frac{S(1,c)}{\sqrt{\rond{N}(c)}} = c_\theta(d) X^{\frac{5}{6}} + \rond{O}\left(X^{\frac{1}{2} +\varepsilon}\right).
\end{equation}
Actually, the first moment of $S(1,c)$ was computed by Livn\'e and Patterson (\cite[Theorem\,1.2]{liv-pat}), and the authors showed that \eqref{intro:eq:3} is true for any $\varepsilon $ greater or equal to $1/4$; moreover, they computed the constant $c_\theta(d)$ when $d$ is square-free. For their purpose, they could avoid the utilization of a complete Bruggeman-Kuznetsov formula, instead they used the simpler version by Goldfeld and Sarnak (\cite{gol-sar:kloo}). This was enough for these authors to compute the asymptotic constant $c_\theta(d)$ when $d$ is square-free, but their result does not give any information on the dependence on the level $d$ of the error term in \eqref{intro:eq:3}. 

Our first main result solves this last problem, by extending the work done by Livn\'e and Patterson. It is a smooth version of \eqref{intro:eq:3}, the main point here being that we have been able to control the dependence on the level:

\begin{thm}\label{intro:thm:2}
Let $g$ be a smooth function with compact support in $[1,2]$. Let $d$ be an Eisenstein integer coprime to $3$. Then there exists a parameter $s(d)$, $0\leqslant s(d)\leqslant \frac{1}{6}$, and a constant $c_\theta(d)$, such that

\begin{align*}
(i)&
\sum_{\substack{c\congru 1 \pmod 3\\c\congru 0\pmod d}}
\frac{S(1,c)}{\sqrt{\rond{N}(c)}} g\left(\frac{\rond{N}(c)}{X}\right) 
= c_\theta(d) \hat{g}(1/6)  X^{\frac{5}{6}} + \rond{O}\left(\sqrt{X} + X^{\frac{1}{2}+s(d)}\log^2 X \frac{\tau(d)}{\rond{N}(d)^{2s(d)}} \right)\\
(ii)&
\sum_{\substack{c\congru 1 \pmod 3\\c\congru 0\pmod d}}
\frac{S(1,c)}{\sqrt{\rond{N}(c)}} g\left(\frac{\rond{N}(c)}{X}\right) 
= \rond{O}\left(\sqrt{X} + X^{\frac{5}{6}}\log^2 X \frac{\tau(d)}{\rond{N}(d)^{\frac{5}{6}}} \right).
\end{align*}
Here, $\hat{g}$ is the Mellin transform of $g$ and $\tau$ is the divisor function. The implied constant in $(ii)$ depends only on the function $g$.
\end{thm}

As a consequence of part $(i)$ of Theorem\,\ref{intro:thm:2}, we obtain the Linnik-Selberg conjecture on average. More precisely, we prove that equation \eqref{intro:eq:3} with $d=1$ is true, in a smooth version: 
 
\begin{cor}\label{intro:cor:1} With the notations of Theorem\,\ref{intro:thm:2}, for any $\varepsilon > 0$,

\begin{equation*}
\sum_{c\congru 1 \pmod 3} \frac{S(1,c)}{\sqrt{\rond{N}(c)}} g\left(\frac{\rond{N}(c)}{X}\right) = c_\theta(d) X^{\frac{5}{6}} + \rond{O}\left(X^{\frac{1}{2} +\varepsilon}\right).
\end{equation*}
\end{cor}

\begin{rque}
The constant $c_\theta(d)$ is computed when $d$ is square-free in \cite{liv-pat}, and could theoretically be computed for a general modulus $d$. This has been partially done in the author's thesis \cite{lou:these}, and we shall come back to this problem in a future work. 
\end{rque}

We now turn to part $(ii)$ of Theorem\,\ref{intro:thm:2}, in which the theta-term is avoided. The explicit dependence on the level $d$ actually allows us to obtain  a non-trivial estimate over arithmetic progressions on average. 
For clarity, let us define the following quantities:
 
\begin{align}
\Sigma (D) 
&= \sum_{\rond{N}(d) \leqslant D} \sum_{\substack{c\congru 1 \pmod 3 \\ c\congru 0 \pmod d}} \frac{S(1,c)}{\sqrt{\rond{N}(c)}} g\left(\frac{\rond{N}(c)}{X}\right), \label{intro:eq:4}\\
\Sigma_\theta (D) 
&= \sum_{\rond{N}(d) \leqslant D} \left\{\sum_{\substack{c\congru 1 \pmod 3\\ c\congru 0 \pmod d}} \frac{S(1,c)}{\sqrt{\rond{N}(c)}} g\left(\frac{\rond{N}(c)}{X}\right)-c_\theta(d) \hat{g}(1/6) X^{\frac{5}{6}}\right\}.\label{intro:eq:5}
\end{align}
On the one side, the Weil bound gives the trivial estimate 
$\Sigma(D) \ll X \log X \log^2 D$. On the other side, 
a consequence of Selberg's eigenvalue conjecture would be 
$\Sigma_\theta(D) \ll D X^{\frac{1}{2}}$, 
and from this would follow at once the estimate

\begin{equation}\label{intro:eq:6}
\Sigma(D) = \hat{g}(1/6) X^{\frac{5}{6}} \sum_{\rond{N}(d)\leqslant D} c_\theta(d) + \rond{O}\left(D X^{\frac{1}{2}}\right).
\end{equation}
It was shown by Livn\'e and Patterson (\cite{liv-pat}) that $c_\theta(d)$ behaves like $1/\sigma(d)$ when $d$ is square-free.  Assuming that $\sum c_\theta(d)$ is finite, it would follow from \eqref{intro:eq:6} that
$\Sigma(D) \ll D X^{\frac{1}{2}}$, whenever $D\gg X^{\frac{1}{3}}$;
this improves on the trivial estimate (coming from the Weil bound) as soon as $D=o\left(\sqrt{X}\right)$. Thus, one would ideally expect an estimate for $\Sigma(D)$ of the form

\begin{equation}\label{intro:eq:7}
\Sigma\left(\frac{\sqrt{X}}{\log^\beta X}\right) \ll \frac{X}{\log X}, \qquad \ec{for some $\beta$} .
\end{equation}
In the second corollary to Theorem\,\ref{intro:thm:2}, we prove a more precise version of \eqref{intro:eq:7}: 

\begin{cor}\label{intro:cor:2} 
Let $\Sigma(D)$ be the quantity defined in \eqref{intro:eq:4}. Let $X\gg 1$ and let $g$ be as in Theorem\,\ref{intro:thm:2}. Then

\begin{equation*}
\Sigma\left( D\right) \ll  D \sqrt{X} + X^{5/6}D^{1/3} \log^3 X. 
\end{equation*}
In particular, \eqref{intro:eq:7} is true for $\beta\geqslant 12$.
\end{cor}
 
Note that in Corollary\,\ref{intro:cor:2}, there is no appearance of the theta-term. In this aspect, Corollary\,\ref{intro:cor:2} is probably not the most effective result, but shall nevertheless enable us to prove the changes of sign of the cubic sums $S(1,c)$, for almost prime moduli. Combining sieve methods and Corollary\,\ref{intro:cor:2}, we will prove the upper bound

\begin{equation}\label{intro:eq:upper}
\left\vert \sum_{\pi\tq c \Rightarrow \rond{N}(\pi)\geqslant X^{1/u}} \frac{S(1,c)}{\sqrt{\rond{N}(c)}} g\left(\frac{\rond{N}(c)}{X}\right) \right\vert 
\leqslant \hat{g}(1) h(u)  \frac{X}{\log X} +\rond{O}\left(\frac{X}{\log^2 X}\right).
\end{equation}
This bound is valid for any $u\geqslant 3$ and $h(u)$ is a rapidly decreasing function, tending to zero as $u$ grows; in particular, a numerical computation shows that $h(60)<10^{-4}$. As we will see in Section\,\ref{sec:sieve}, Corollary\,\ref{intro:cor:2} plays an essential role in the proof of \eqref{intro:eq:upper}.

Finally, we will prove the counterpart to \eqref{intro:eq:upper}; more precisely, we will obtain, in a similar way to \cite{fou-mic:chgmt-signe}, the lower bound

\begin{equation}\label{intro:eq:lower}
\sum_{\pi\tq c \Rightarrow \rond{N}(\pi)>X^{1/u}} \left\vert \frac{S(1,c)}{\sqrt{\rond{N}(c)}} g\left( \frac{\rond{N}(c)}{X}\right) \right\vert 
\geqslant C_{\ec{abs}} \hat{g}(1) \frac{X}{\log X}.
\end{equation}
The main ingredient of the proof of \eqref{intro:eq:lower} is the vertical Sato-Tate law (Theorem\,\ref{intro:thm:vert}). 
This last bound \eqref{intro:eq:lower} is valid for any $u\geqslant 3$, and the value of $C_{\ec{abs}}$ is given by $C_{\ec{abs}}= 0.015$. 
Theorem\,\ref{intro:thm:1} is now proved, as a consequence of \eqref{intro:eq:upper} and \eqref{intro:eq:lower}. 

\begin{rque} 

(1) As already mentioned before, our goal in this paper is to give an explicit value for the constant $u$ of Theorem\,\ref{intro:thm:1}, and the problem of improving the value of $u$ lies beyond the scope of this paper. However, it is a very interesting question, which can be tackled in a number of ways: in \cite{fou-mic:chgmt-signe}, the authors used algebraic geometric methods in order to optimize the constant $C_{\ec{abs}}$ in \eqref{intro:eq:lower}; moreover, one could expect to improve Corollary\,\ref{intro:cor:2}, by proving that $\Sigma\left(X^\alpha\right) \ll \frac{X}{\log X}$ for some $\alpha>1/2$; one could also refine the sieve argument, as it has been done in the case of Kloosterman sums (see \cite{siv:kloo}); another interesting work in this direction is \cite{mat:kloo}. 

(2)  The core of the paper is devoted to the proof of Theorem\,\ref{intro:thm:2}; the main difficulty lies in the machinery of metaplectic forms and in the Bruggeman-Kuznetsov formula for imaginary quadratic field. In this regard, we considerably benefit from the works of Bruggeman and Motohashi (\cite{bru-mot:fourth}) and Lokvenec-Guleska (\cite{gul:these}), where the {\it Kloosterman sum formula} and the {\it spectral sum formula} have been precisely formulated for quadratic imaginary fields.  

(3) Unfortunately, the fact that our proof is based on the theory of metaplectic forms on $\Q(e^{2i\pi/3})$ does not allow us to obtain results on the cubic exponential sums $S(1,c)$, for $c$ running over the rational integers. It is actually not clear what would be the analogue of Theorem\,\ref{intro:thm:2} over $\Z$. We refer to \cite{pat:distrib-1} for a discussion and a conjecture concerning the asymptotic behavior of exponential sums, and leave as an open question wether or not it is possible to derive from Theorem\,\ref{intro:thm:1} an analogous statement for moduli in $\Z$.   

(4) As it has been implicitly observed previously, Corollary\,\ref{intro:cor:2} states that in average, we can assume that there is no exceptional eigenvalue, rendering therefore useless for the proof of Theorem\,\ref{intro:thm:1} the information on the exceptional spectrum obtained in Theorem\,\ref{intro:thm:2} $(i)$. It would be very interesting to develop a sieve method that takes into account the exceptional spectrum, and reveals its contribution in Theorem\,\ref{intro:thm:1}, under the form of a bias.
\end{rque}

We end the introduction with a brief outline of the rest of the paper. In Section\,\ref{sec:Kloo}, we recall the link between cubic exponential sums and Kloosterman sums twisted by the cubic Legendre symbol. In Section\,\ref{sec:meta}, we introduce the cubic metaplectic forms and recall some facts about the discrete spectrum of the metaplectic group. In Section\,\ref{sec:spectrum}, we state the Bruggeman-Kuznetsov formula for the imaginary quadratic field $\Q(\omega)$, and prove some estimates related to the discrete spectrum of the metaplectic group. The proofs of Theorem\,\ref{intro:thm:2} and of its corollaries are finally given in Section\,\ref{sec:proof}. The proof of Theorem\,\ref{intro:thm:1} is then obtained in the last two sections: the upper bound \eqref{intro:eq:upper} is proved in Section\,\ref{sec:sieve} and the lower bound \eqref{intro:eq:lower} in Section\,\ref{sec:sato-tate}.\\

\begin{notations}
We denote by $k$ the field $k=\Q(\omega)$, with $\omega=e^{2i\pi/3}$, $R$ being the ring of Eisenstein integers $\Z[\omega]$. For $z\in\C$, $e(z)=\exp(\ec{Tr}_{k/\Q}(z))=\exp(2i\pi (z+\barre{z}))$, where $z\mapsto \barre{z}$ is the complex conjugation. The sums will be taken over $R$ and $\barre{x}$ shall be an inverse of $x\pmod r$, for some $r\in R$ given by the context. The function $\omega(c)$ will be the number of distinct prime factors of the integer $c$ and, in Section\,\ref{sec:sieve}, it will also represent the Buchstab function. Finally, $\delta_{\alpha,\beta}$ is $1$ or $0$, according to if $\alpha=\beta$ or not. \\
\end{notations}

\begin{ack}
This article is based on Chapter\,$2$ and Chapter\,$4$ of the author's PhD thesis \cite{lou:these}. I sincerely thank my supervisors, Professor Samuel James Patterson and Professor Philippe Michel, for introducing me to the theory of exponential sums and for their support and encouragement. I also thank Professor Valentin Blomer for his advice and comments on this paper. I want to thank the Ecole Polytechnique F\'ed\'erale de Lausanne and the Universit\'e de Montpellier\,2, where part of this work has been done, for excellent working conditions. 
\end{ack}

\section{Cubic and geometric Kloosterman sums}\label{sec:Kloo}

The link between the cubic exponential sums $S(a,c)$ defined in \eqref{intro:eq:0} and the theory of automorphic forms is given by the cubic Kloosterman sums, i.e. the Kloosterman sums twisted by the cubic legendre symbol. The cubic Kloosterman sums will appear as Fourier coefficients of Poincar\'e series, and thereby provide a link with the spectral theory of automorphic forms. Let $\left( \frac{\pt}{\pt}\right)_{\!3}$ be the cubic residue symbol. For $m,n,c\in \Z[\omega]$ with $\ec{gcd}(c,3)=1$, the cubic Kloosterman sum is defined as

\begin{equation}\label{Kloo:eq:1}
K_3(m,n,c)=\sum_{\substack{x\pmod c\\ x \barre{x} \congru 1 \pmod c}} \left(\frac{x}{c}\right)_{\!3} e\left(\frac{mx+n\barre{x}}{c}\right).
\end{equation}
The following relation is due to several authors (Duke and Iwaniec, Livn\'e, Katz). The version we give here is taken from \cite{pat:distrib}, Theorem\,$3.1$. 

\begin{prop}\label{Kloo:prop}
Let $a,c\in \Z[\omega]$. Assume that $\emph{gcd}(a,c)=1$ and $\emph{gcd}(c,3)=1$. Then, 
$S(a,c)=K_3(a,a,c)$.
\end{prop}

The three-dimensional hyperbolic space is usually represented as the half-space $\mathbb{H}=\C\times \R_+^\times$. We can embed it in the Hamiltonian quaternions by identifying $\sqrt{-1}\in\C$ with $\hat{i}$ and $w=(x+iy,v)\in \mathbb{H}$ with $x+y\hat{i}+v\hat{k}$, where $1,\hat{i},\hat{j},\hat{k}$ are the standard unit quaternions. 
The group ${\rm SL}_2(\C)$ acts on $\mathbb{H}$ by 

\begin{equation*}
\begin{pmatrix} a & b \\ c & d\end{pmatrix} w = (a w+b) (cw+d)^{-1}.
\end{equation*}

The ${\rm SL}_2(\C)$-invariant measure is $dV (w) = v^{-3} \,dm(z) dv$, where $dm (z)$ is the standard Lebesgue measure on $\C$. 
Of first importance for us are the subgroups of ${\rm SL}_2(\Z[\omega])$ defined by

\begin{align}
\Gamma_2 \ \
&=\{\gamma\in {\rm SL}_2(\Z[\omega]) \, : \, \exists g\in {\rm SL}_2(\Z), \gamma\congru g \pmod 3\}, \label{Kloo:eq:2}\\
\Gamma_1 \ \
&=\{\gamma\in {\rm SL}_2(\Z[\omega]) \, : \, \gamma\congru 1 \pmod 3\}, \label{Kloo:eq:3}\\
\Gamma_0(d) 
&= \{\gamma\in {\rm SL}_2(\Z[\omega]) \, : \, \gamma\congru 
\left(\begin{smallmatrix} * & * \\ 0 & * \end{smallmatrix}\right)  \pmod d\}. \label{Kloo:eq:4}
\end{align}
The Kubota symbol $\kappa$ can now be introduced. It is defined on $\Gamma_1$ by

\begin{equation*}
\kappa (\gamma) = \begin{cases} 
\left(\frac{c}{a}\right)_{\!3} & \ec{if $c\neq 0$}\\
1 &\ec{if $c=0$},
\end{cases} \qquad \ec{ where } \gamma=\begin{pmatrix}a&b\\ c&d\end{pmatrix}\in \Gamma_1.
\end{equation*}
This definition is then extended to $\Gamma_2$ by defining $\kappa$ trivially on ${\rm SL}_2(\Z)$. More precisely, for any $\gamma_2\in\Gamma_2$, there exists $g\in {\rm SL}_2(\Z)$ and $\gamma_1\in \Gamma_1$ such that $\gamma_2=g \gamma_1$, and we define 

\begin{equation}\label{Kloo:eq:5}
\kappa(\gamma_2) = \kappa(\gamma_1).
\end{equation}
The starting point of the theory of metaplectic forms is a short but significant paper of Kubota (\cite{kub:1}), in which he proved that $\kappa$ is a group homomorphism on $\Gamma_1$. Actually, Kubota proved that $\kappa$ is a morphism on the subgroup of $\Gamma_1$ consisting in matrices congruent to $1$ modulo $9$, but this last condition can be dropped, and it can be proved that the definition of $\kappa$ on $\Gamma_2$ extends $\kappa$ to a group homomorphism from $\Gamma_2$ into the cubic roots of unity of $k$ (see \cite{pat:theta}, p.127). 

Let $\sigma\in SL_2(R)$ and let $\Gamma$ be a subgroup of $\Gamma_2$. Define 

\begin{equation*}
\Gamma_\sigma =\left\{ \gamma\in \Gamma\,:\, \gamma \left(\sigma^{-1}(\infty)\right) = \sigma^{-1}(\infty)\right\}.
\end{equation*}
One can show that for any $\sigma \in SL_2(R)$, there exists a lattice $\Lambda_\sigma \in R$ such that 

\begin{equation*}
\Gamma_\sigma = \sigma^{-1}  \begin{pmatrix} 1&\Lambda_\sigma \\0&1\end{pmatrix} \sigma.
\end{equation*}
For $\Lambda$ a lattice in $R$, we denote by $\Lambda^\dual$ the dual lattice to $\Lambda$ with respect to $e$.
A cusp $\sigma^{-1}(\infty)$ of $\Gamma$, with $\sigma^{-1}\in {\rm SL}_2(\Z[\omega])$, is called {\it essential} if $\kappa\vert_{\Gamma_\sigma}=1$.

\begin{defi}\label{Kloo:defi}
Let $\Gamma$ be a subgroup of $\Gamma_2$ and let $\sigma^{-1}(\infty)$ and $\tau^{-1}(\infty)$ be two essential cusps of $\Gamma$. Let $m\in \Lambda_\sigma^\dual$ and $n\in\Lambda_\tau^\dual$. The geometric Kloosterman sum is defined by

\begin{equation*}
K_{\sigma,\tau}(m,n,c)= \sum_{\substack{a \pmod {\Lambda_\sigma c} \\ d \pmod {\Lambda_\tau c} \\ \left(\begin{smallmatrix}a&* \\ c&d \end{smallmatrix}\right) \in \Gamma}} \kappa(\gamma) e\left( \frac{am+nd}{c}\right).
\end{equation*}
\end{defi}

Let $\omega(c)$ be the number of prime divisors of $c$. Then the geometric Kloosterman sums satisfy the individual bound  

\begin{equation}\label{Kloo:eq:Weil}
\left\vert K_{\sigma,\tau}(m,n,c) \right\vert \leqslant 2^{\omega(c)} \rond{N}\left( \ec{gcd}(m,n,c)\right) \rond{N}(c)^{1/2}.
\end{equation}
We also have the following "twisted multiplicativity": generally, if $f$ is a rational function with integral coefficients, then the exponential sum 

\begin{equation*} 
S(f,c) =\sum_{\substack{x \pmod c\\ f(x)\neq \infty}}e\left(\frac{f(x)}{c}\right) \left(\frac{x}{c}\right)_{\!\!3},
\end{equation*}
satisfies, for every $c_1,c_2 \in R$, such that $\ec{gcd}(c_1,c_2)=1$, the twisted multiplicativity

\begin{equation}\label{Kloo:eq:mult}
S(f,c_1c_2) = \left(\frac{c_1}{c_2}\right)_{\!\!3} \left(\frac{c_2}{c_1}\right)_{\!\!3} S(f_2,c_1) S(f_1,c_2),
\end{equation}
where $f_i(x)=c_i^{-1} f(c_i x)$, with $c_i^{-1} c_i \congru 1 \pmod {c_{3-i}}$ for $i=1,2$. 
We refer for example to \cite[Proposition\,5.1]{liv-pat} for a proof of \eqref{Kloo:eq:Weil} and \eqref{Kloo:eq:mult}.

In order to come back to the arithmetical setting, one has to fix the group $\Gamma$ and the two cusps $\sigma^{-1}(\infty)$ and $\tau^{-1}(\infty)$. Let $d$ be a primary integer, i.e. $d\congru 1 \pmod 3$. According to this, we shall work with the groups 

\begin{align}
\Gamma_1^- &= \langle \Gamma_1,-Id\rangle, \label{Kloo:eq:6}\\
\Gamma_d &= \langle \Gamma_1,-Id\rangle \cap \Gamma_0(d)
\subset \Gamma_1^-.\label{Kloo:eq:7}
\end{align}
For $d=1$ or $d=2$, the group $\Gamma_d$ is not equal to the group defined in \eqref{Kloo:eq:2} and \eqref{Kloo:eq:3}, but this should not cause any confusion. The reason for including $-Id$ in the group is that it will allow us to work with even functions in the Bruggeman-Kuznetsov formula, as in \cite{bru-mot:fourth}; the case of odd functions has been worked out in \cite{gul:these} but led there to complications in the sum formulas. 
We shall also work with the matrices

\begin{equation}\label{Kloo:eq:8}
\sigma^{-1}= \begin{pmatrix} 1&0 \\ 0&1 \end{pmatrix} \qquad \ec{and } \qquad \tau^{-1}=\begin{pmatrix} d-1 & d-2 \\ d&d-1\end{pmatrix}.
\end{equation}
We remark that for any primary $d$, the cusps $\sigma^{-1}(\infty)$ and $\tau^{-1}(\infty)$ are two essential cusps with respect to the group $\Gamma_d$, but they are not $\Gamma_d$-equivalent.

\begin{lemme}\label{Kloo:lemme}
With the notations of \eqref{Kloo:eq:7} and \eqref{Kloo:eq:8}, the geometric Kloosterman sums defined in Definition\,\ref{Kloo:defi} satisfy

\begin{align*}
(i)&
& K_{\sigma,\tau}(m,n,c) &= 
\begin{cases}
K_3(m,n,c) & \ec{ if } c\congru \pm 1\pmod 3, c \congru 0 \pmod d \\
0& \ec{ otherwise}.
\end{cases}\\
(ii)&
& K_{\sigma,\sigma}(m,n,c)&=  K_{\tau,\tau}(m,n,c) = 0, \quad \ec{ if $c$ is not divisible by $d$}.
\end{align*}
\end{lemme}

\begin{proof}
We have $\Lambda_\sigma=\Lambda_\tau= 3 \Z[\omega]$. Then Definition\,\ref{Kloo:defi} reads

\begin{equation*}
K_{\sigma,\tau}(m,n,c) = \sum_{\substack{a\pmod{3c}\\ d\pmod{3c} \\ \left(\begin{smallmatrix} a&* \\ c&d \end{smallmatrix}\right) \tau\in \Gamma}} \barre{\kappa}\left( \begin{pmatrix} a&*\\c&d\end{pmatrix} \begin{pmatrix} d-1&2-d\\-d&d-1 \end{pmatrix}\right) e\left(\frac{ma+nd}{c}\right).
\end{equation*}
With our choice $\Gamma=\Gamma_d$, the condition $\left(\begin{smallmatrix} a&* \\ c&d \end{smallmatrix}\right) \tau\in \Gamma$ means $a\congru d \congru 0 \pmod 3$, $c\congru \pm 1\pmod 3$ and $c\congru 0\pmod d$. Then, from the definition of $\kappa$ on $\Gamma_2$, one knows that $\kappa(\gamma)=\kappa(\gamma' \gamma)$, for all $\gamma \in \Gamma_2, \gamma'\in{\rm SL}_2(\Z)$; using this, one shows that

\begin{equation*}
\barre{\kappa}\left( \begin{pmatrix} a&*\\c&d\end{pmatrix} \begin{pmatrix} d-1&2-d\\-d&d-1 \end{pmatrix}\right)
= \barre{\left(\frac{a}{c}\right)_{\!\!3}}.
\end{equation*}
This proves $(i)$. The proof of $(ii)$ is similar, and we omit it.
\end{proof}

\section{Cubic metaplectic forms}\label{sec:meta}

In this section, we recall some classical definitions and results about cubic metaplectic forms. These results will be used in the next section, in order to state the Bruggeman-Kuznetsov formula for the field $\Q(\omega)$. Since we will make use of the work of Bruggeman, Motohashi and Lokvenec-Guleska (\cite{bru-mot:fourth} and \cite{gul:these}) in Section\,\ref{sec:spectrum}, we will follow their notations in this section as well. Moreover, special emphasis will be put on the residual part of the spectrum of the Laplacian operator, which is non-trivial in our context; this corresponds to the so called theta-term, which was introduced by Kubota in \cite{kub:2}, and explicitely treated by Patterson in \cite{pat:theta}.

Let $G$ be the group $G= {\rm PSL}_2(\C)$; its Iwasawa decomposition is given by

\begin{equation}
G = \barre{N} \barre{A} \barre{K},
\end{equation}
where $\barre{N}$, $\barre{A}$ and $\barre{K}$ are the projective images in $G$ of the following subgroups $N$,  $A$ and $K$ of ${\rm SL}_2(\C)$:

\begin{align*}
N & =\{n[z] \,:\, z\in \C\} & &\ec{where } n[z]=\begin{pmatrix} 1& z \\ 0&1\end{pmatrix},\\
A & =\{ a[r] \,:\, r>0\} & &\ec{where } a[r]=\begin{pmatrix} \sqrt{r}& 0 \\ 0&\sqrt{r}^{-1}\end{pmatrix},\\
K &=SU(2) & &\ec{with elements } k(\alpha,\beta)=\begin{pmatrix} \alpha & \beta \\-\barre{\beta} & \barre{\alpha}\end{pmatrix}.
\end{align*}

The real Lie algebra $\id{sl}_2$ of $G$ is generated by the six elements

\begin{align*}
& \boldsymbol{H_1} =
\frac{1}{2} \begin{pmatrix} 1&0\\0&-1\end{pmatrix}, &
& \boldsymbol{V_1} =
\frac{1}{2} \begin{pmatrix} 0&1\\1&0\end{pmatrix}, &
& \boldsymbol{W_1} =
\frac{1}{2} \begin{pmatrix} 0&1\\-1&0\end{pmatrix}, \\
& \boldsymbol{H_2} =
\frac{1}{2} \begin{pmatrix} i&0\\0&-i\end{pmatrix}, &
& \boldsymbol{V_2} =
\frac{1}{2} \begin{pmatrix} 0&i\\-i&0\end{pmatrix}, &
& \boldsymbol{W_2} =
\frac{1}{2} \begin{pmatrix} 0&i\\i&0\end{pmatrix}.
\end{align*}
The complexification $\id{g}$ of $\id{sl}_2$ can be seen as the set of all left-invariant differential operators. We have $\id{g} \iso \id{sl}_2 \oplus \id{sl}_2$, and one shows that the two elements

\begin{align*}
\Omega_\pm = \frac{1}{8} \left((\boldsymbol{H_1}\mp i \boldsymbol{H_2})^2 + (\boldsymbol{V_1}\mp i \boldsymbol{V_2})^2 - (\boldsymbol{W_1}\mp i \boldsymbol{W_2})^2 \right)
\end{align*}
generate over $\C$ the center $\rond{Z}(\id{g})$ of the universal enveloping algebra $\rond{U}(\id{g})$. 

The real Lie algebra of $\barre{K}$ is generated by $\boldsymbol{H_2}$, $\boldsymbol{W_1}$ and $\boldsymbol{W_2}$. Its complexification $\id{k}$ is of dimension two, and the center $\rond{Z}(\id{k})$ of $\rond{U}(\id{k})$ is generated by 

\begin{equation*}
\Omega_\id{k} = \frac{1}{2} \left( \boldsymbol{H_2}^2 + \boldsymbol{W_1}^2 +\boldsymbol{W_2}^2\right).
\end{equation*}

We need to study automorphic forms which are not $K$-invariant. In this paragraph, we give some facts about $L^2(K)$. An orthogonal basis of $ L^2(K)$ is known to be given by $\{ \Phi^l_{p,q}\, :\, l\geqslant 0, \tq p\tq \leqslant l, \tq q \tq \leqslant l\}$. These functions are constructed as matrix coefficients of some representation of $K$; we refer to the discussion in \cite[\S\,2.2]{gul:these} and in \cite[\S\,3]{bru-mot:fourth} for more details. Since they satisfy

\begin{equation*}
\Omega_\id{k} \Phi^l_{p,q} = -\frac{1}{2} (l^2+l) \Phi^l_{p,q}, \qquad \boldsymbol{H_2} \Phi^l_{p,q} = -iq \Phi^l_{p,q},
\end{equation*} 
one can reorganize them and obtain the following decomposition:

\begin{equation*} 
L^2(K) = \barre{\bigoplus_{\substack{l,q\\ \tq q\tq \leqslant l}}L^2(K;l,q)}, 
\end{equation*}
with

\begin{align*}
\qquad L^2(K; l,q) 
&= \bigoplus_{\tq p\tq \leqslant l} \C \Phi^l_{p,q}\\
&=\left\{ f\in L^2(K) L^2(K) \, : \, \Omega_\id{k} f = -\frac{1}{2} (l^2+l) f, \boldsymbol{H_2} f = -iq f\right\}.
\end{align*}

Finally, we now give a model for irreducible representations of $G$. We define the functions $\phi_{l,q}(s,p)$ on $G$ by
$\phi_{l,q}(s,p)(na[r]k)=r^{1+s}\Phi^l_{p,q}(k)$. Let $H(s,p)$ be the space generated by all finite linear combinations of $\phi_{l,q}(s,p)$, and let $H^2(s,p)$ be the completion of $H(s,p)$ in $L^2(K)$. Then $H^2(s,p)$ is $\id{g}$-invariant and irreducible for the values of $s$ and $p$ we will be interested in. Let us mention that

\begin{align}
\boldsymbol{H_2} \phi_{l,q}(s,p) 
&= -iq \phi_{l,q}(s,p),\\
\Omega_\id{k} \phi_{l,q}(s,p) 
&= -\frac{l^2+l}{2} \phi_{l,q}(s,p),\\
\Omega_\pm \phi_{l,q}(s,p) 
&= \frac{1}{8} \left( (s\mp p)^2-1\right) \phi_{l,q}(s,p).
\end{align}  

Automorphic forms on the hyperbolic upper half space can be seen as automorphic forms on ${\rm SL}_2(\C)$ which are invariant by the action of the maximal compact subgroup $K$ of ${\rm SL}_2(\C)$. Over $\Q$, the Kuznetsov sum formula is an equality between Kloosterman sums and spectral elements involving Maass forms and holomorphic forms. This generalizes for a number field by considering automorphic forms of any $K$-type. 

In the rest of this section we shall consider a discrete subgroup $\Gamma$ of $SL_2(R)$. 
Let $L^2\left( \Gamma \backslash G,\kappa\right)$ be the space of square-integrable functions on $G$ which satisfy 

\begin{equation*}
f(\gamma g) = \kappa(\gamma) f(g) \qquad \ec{for all } \gamma \in \Gamma.
\end{equation*} 
Under the action of $G$ by the regular representation, $L^2\left( \Gamma \backslash G,\kappa\right)$ decomposes into irreducible unitary representations with finite multiplicities, say 

\begin{equation}
L^2\left( \Gamma \backslash G,\kappa\right)=\barre{\bigoplus V}.
\end{equation}
As we saw before, the irreducible representations of $G$ are known; they are given by the $H^2(s,p)$, for $s\in i\R$ or $s\in]0,1[$. 
Moreover the operator $\Omega_{\pm}$ acts on the subspace $H^K(s,p)$ of $K$-finite vectors of $H^2(s,p)$ as multiplication by the scalar $1/8 \left((s\mp p)^2-1\right)$. 
In consequence, each space $V$ can be decomposed as an infinite direct sum of irreducible representations with respect to the action of $K$. If $V$ is isomorphic to $H(s,p)$, then the decomposition of $V$ is given by

\begin{equation}
V=\bigoplus_{\substack{l\geqslant \tq p\tq \\ \tq q\tq <l}} V_{l,q}.
\end{equation}
The space $V_{l,q}$ is one dimensional, generated by the function $\phi_{l,q}(s,p)$. 
According to this, one defines the space of metaplectic forms of a given $K$-type.

\begin{defi} A metaplectic form of $K$-type $(l,q)$ with respect to the group $\Gamma$ is a function $f\,:\, G\too \C$ satisfying

\begin{align*}
f(\gamma g) &= \kappa(\gamma) f(g), \qquad \forall \gamma\in\Gamma,\\
\Omega_\pm f &= \lambda f, \qquad \ec{for some $\lambda \in \C$},\\
\gras{H}_2 f &= -iq f, \qquad  
\Omega_{\id{k}} f =  -\frac{l^2+l}{2} f.
\end{align*}
\end{defi}
One can show that the value $\lambda$ has to be of the form $\lambda= \frac{1}{8}\left((s\mp p)^2-1\right)$, for some $s$ and $p$ (see \cite{gul:these}, Lemma\,3.2.2). 

\begin{defi} 
Let $f$ be a metaplectic form of eigenvalue $1/8 \left((s-p)^2-1\right)$ with respect to $\Omega_\pm$. 
The couple $(s,p)$ is then called the spectral parameter of $f$. If $s\in ]0,1]$, then $p=0$ and $(s,p)$ is said to be an exceptional spectral parameter.  
\end{defi}
 
\begin{rque}
In particular, metaplectic forms of $K$-type $(0,0)$ have necessarily a spectral parameter of the form $(s,0)$; they are actually functions defined over $\mathbb{H}$. 
\end{rque}

We now recall that the way of expanding  a metaplectic form in its Fourier series is independent of its $K$-type. On the space of functions $f\in C^\infty(G)$ such that, for some $\sigma >0$, 

\begin{equation*}
f(na[v]k)=\rond{O}\left( v^{1+\sigma}\right), \qquad \ec{ as } v\to 0,
\end{equation*}
define the operator $\rond{A}_m$ by

\begin{equation*}\rond{A}_mf (g) =\int_N\barre{\chi}_m(n) f(wng) \,dn,
\end{equation*}
where $w=\left(\begin{smallmatrix} 0&-1 \\ 1&0\end{smallmatrix}\right)$ and $\chi_u(n)=e(un)$, for $n=n[z]$.
Now, if $(s,p)$ is the spectral parameter corresponding to some irreducible representation $V$ as above, and if the isomorphism between $H(s,p)$ and $V$ is denoted by $T$, then
for any element $\phi_{l,q}(s,p)$ of $H(s,p)$, the $m^{\ec{th}}$ Fourier coefficient of $T\phi_{l,q}(s,p)$ at a cusp $\sigma^{-1}(\infty)$ is a multiple of $\rond{A}_m \phi_{l,q}(s,p)$ by some constant $\rho_{s,p}(\sigma,m)$. This number is the Fourier coefficient of the representation $V$, i.e. we have

\begin{equation}
T=\sum_{m\in \Lambda_\sigma^\dual} \rho_{s,p}(\sigma,m) \rond{A}_m.
\end{equation}
In particular, for any $l\geqslant \tq p\tq$ and any $-l \leqslant q \leqslant l$, the metaplectic forms $T\phi_{0,0}(s,p)$ and $T\phi_{l,q}(s,p)$ of $V$ have the same Fourier coefficients. 

Let $f$ be a metaplectic form of $L^2\left(\Gamma \backslash G,\kappa\right)$ with spectral parameter $(s,p)$. Then $f$ is said to be cuspidal if $\rho(\sigma,0)=0$ for any cusp $\sigma^{-1}(\infty)$ of $\Gamma$. An important feature of metaplectic forms is that there exists non-cuspidal form. According to Selberg's theory, 
 such forms are residues of Eisenstein series. We will give more details about this below. 

As it has already been mentioned above, remark that $p=0$ allows us to choose $l=q=0$. In this way one gets forms on $\mathbb{H}$. In the non-metaplectic context, the Selberg conjecture predicts the non-existence of exceptional spectral parameter. We now give a result concerning the analog of Selberg's conjecture for cubic metaplectic forms. 

\begin{thm}\label{meta:thm:1}
Let $d$ be a primary Eisenstein integer. Let $\Gamma_d$ be defined by \eqref{Kloo:eq:7}. There exists a parameter $0\leqslant s(d)\leqslant \frac{1}{6}$ such that if $(s,0)$ is the exceptional spectral parameter of a metaplectic form with respect to $\Gamma_d$, then $s$ belongs to $]0,s(d)] \cup \{\frac{1}{3}\}$. 
\end{thm}

\begin{proof}
This theorem is proved by combining two facts: the upper bound $s(d) \leqslant \frac{1}{2}$ of Jacquet-Gelbart \cite{gel-jac:gl2gl3} for non-constant non-metaplectic forms 
and the cubic Shimura correspondence. The latter is a theorem of Flicker (see \cite{fli:cov}) about global automorphic representations of the metaplectic group and of $GL_2$. Given a metaplectic form $\tild{f}$ with spectral parameter $(\tild{s},0)$, there exists an automorphic form $f$ with spectral parameter $(s,0)$. At the archimedean place, the cubic Shimura correspondence states that $\tilde{s}=\frac{1}{3} s$. 
\end{proof}

\begin{defi} 
Let $\sigma^{-1}(\infty)$ be an essential cusp of $\Gamma$ . Let $p,l,q\in \Z$, $\tq p\tq, \tq q\tq \leqslant l$, and $p\congru l\congru q \pmod 1$. For $\Re(s)>1$, the Eisenstein series $E_\sigma(s,p,l,q;g)$ is 
defined by

\begin{equation*}
E_\sigma(s,p,l,q;g)=\sum_{\gamma \in\Gamma_\sigma \backslash \Gamma} \barre{\kappa}(\gamma) \phi_{l,q}(s,p)\left(\sigma\gamma g\right).
\end{equation*}
\end{defi}
It admits a Fourier expansion at any cusp of $\Gamma$, but we shall need to work only with essential cusps. If $\tau^{-1}(\infty)$ is an essential cusp of $\Gamma$, then 

\begin{equation}\begin{split}
&E_\sigma(s,p,l,q;  \tau^{-1}g)=\delta_{\sigma,\tau}\phi_{l,q}(s,p)(g)\\
&\qquad + \frac{\pi  (-1)^{p-\tq p\tq}}{\ec{Vol}(\Lambda_\tau)} \frac{\Gamma(l+1-s) \Gamma(\tq p\tq +s)}{\Gamma(l+1+s) \Gamma(\tq p\tq +1-s)} \psi_{\sigma,\tau}(s,0,p) \phi_{l,q}(-s,-p)(g)\\
&\qquad + \frac{1}{\ec{Vol}(\Lambda_\tau)} \sum_{0\neq m\in\Lambda_\tau^\dual} \psi_{\sigma,\tau}(s,m,p) \rond{A}_m\phi_{l,q}(s,p)(g),
\end{split}\end{equation}
where the coefficients $\psi_{\sigma,\tau}(s,m,p)$ are Dirichlet series formed by cubic Gau\ss{} sums, i.e.  

\begin{equation*}
\begin{split}
\psi_{\sigma,\tau}(s,m,p)&=\sum_{c\neq 0} \rond{N}(c)^{-(1+s)} \left(\frac{c}{\tq c\tq}\right)^{2p} \\
&\sum_{\sigma^{-1} \left(\begin{smallmatrix} a&b \\ c&d \end{smallmatrix}\right) \tau \in \Gamma_\sigma \backslash \Gamma / \Gamma_\tau} \barre{\kappa}\left( \sigma^{-1} \begin{pmatrix}a&b \\ c&d \end{pmatrix}\tau\right) e\left(\frac{md}{c}\right).
\end{split}
\end{equation*}
The properties needed for our applications are listed below. Consider the function $E_\sigma(s,p,l,q;g)$ as a function of the variable $s$. Then  

\begin{itemize}
\item[(i)] $E_\sigma(s,p,l,q;g)$ possesses a meromorphic continuation to $\C$ and a functional equation relating $E_\sigma(1+s,p,l,q;g)$ and $E_\sigma(1-s,p,l,q;g)$,
\item[(ii)] $E_\sigma(s,p,l,q;g)$ is holomorphic if $p\neq 0$,
\item[(iii)] $E_\sigma(s,p,l,q;g)$ has poles at $s=\frac{-1}{3}$ and $s= \frac{1}{3}$, if $p=0$.
\end{itemize}
Actually, taking the residue of Eisenstein series gives square-integrable non cuspidal automorphic forms. They are functions on $\mathbb{H}$, eigenfunctions of the Laplacian with eigenvalue $1-s^2$. This is the minimal eigenvalue of the Laplacian, and in our case it is $1-s^2=8/9$. 

\begin{defi} In the half-plane $\Re(s)\geqslant 0$, the theta function associated to the essential cusp $\sigma^{-1}(\infty)$ of $\Gamma$ is defined as

\begin{equation*}
\theta_\sigma((l,q),g)=\ec{Res}_{s=1/3}\left(E_\sigma(s,0,l,q;g)\right).
\end{equation*}
It is a square integrable non cuspidal metaplectic forms of spectral parameter $(\frac{1}{3},0)$ and of $K$-type $(l,q)$. 
\end{defi}

From the Fourier expansion of $E_\sigma(s,p,l,q;g)$, one gets 

\begin{equation*}
\begin{split}
\theta_\sigma((l,q),\tau^{-1}g)
&= \frac{\pi}{\ec{Vol}(\Lambda_\tau)} \frac{\Gamma(l+2/3) \Gamma(1/3)}{\Gamma(l+4/3) \Gamma(2/3)} \rho_{\theta_\sigma}(\tau,0) \phi_{l,q}(-1/3,0)(g)\\
&+ \frac{1}{\ec{Vol}(\Lambda_\tau)} \sum_{0\neq m\in\Lambda_\tau^\dual} \rho_{\theta_\sigma}(\tau,m) \rond{A}_m\phi_{l,q}(1/3,0)(g),
\end{split}
\end{equation*}
where 

\begin{equation*}
\rho_{\theta_\sigma}(\tau,m) =\ec{Res}_{s=1/3} \left( \psi_{\sigma,\tau}(s,m,0)\right).
\end{equation*}

Let $L^{2,\ec{res}}\left(\Gamma\backslash G, \kappa\right)$ be the space generated by the theta series $\theta_\sigma((l,q),g)$, where $l\in \N$, $q\in \Z$, $\tq q\tq \leqslant l$ and where $\sigma^{-1}(\infty)$ runs over a set of inequivalent essential cusp of $\Gamma$. The set $L^{2,\ec{cusp}}\left(\Gamma\backslash G, \kappa\right)$ is defined as the space generated by cuspforms. We conclude this section by stating the spectral decomposition theorem:

\begin{thm}\label{meta:thm:2}
Let $L^{2,disc}\left(\Gamma\backslash G, \kappa\right)$ be the direct sum of the invariant irreducible subspaces of ${\rm L}^2\left(\Gamma\backslash G, \kappa\right)$. Then
$L^{2,disc}\left(\Gamma\backslash G, \kappa\right)$ is the direct sum of ${\rm L}^{2,\ec{res}}\left(\Gamma\backslash G, \kappa\right)$ and ${\rm L}^{2,\ec{cusp}}\left(\Gamma\backslash G, \kappa\right)$ and, if $L^{2,\ec{cont}}\left(\Gamma\backslash G, \kappa\right)$ is the orthogonal complement to \\
$L^{2,disc}\left(\Gamma\backslash G, \kappa\right)$, we have

\begin{equation}
L^2\left(\Gamma\backslash G, \kappa\right) =
L^{2,\ec{res}}\left(\Gamma\backslash G, \kappa\right) \oplus
L^{2,\ec{cusp}}\left(\Gamma\backslash G, \kappa\right) \oplus
L^{2,\ec{cont}}\left(\Gamma\backslash G, \kappa\right) .
\end{equation}
\end{thm}

\section{On the spectrum of the metaplectic group}\label{sec:spectrum}

In this section we state the Bruggeman-Kuznetsov formula for $\Q(\omega)$ and obtain some estimates related to the discrete spectrum of the metaplectic group. Let $J_s(z)$ be the usual Bessel function. Let us define  

\begin{align*}
\rond{J}_{s,p}(z)&= J_{s-p} (z) J_{s+p}(\barre{z}) \\
&= \left\vert \frac{z}{2}\right\vert^{2s} \left(\frac{z}{\tq z\tq}\right)^{-2p} \sum_{m,n\geqslant 0} \frac{(-1)^{m+n} \left(z/2\right)^{2m} \left(\barre{z}/2\right)^{2n}}{m! n! \Gamma(s-p+m+1) \Gamma(s+p+n+1)}\\
\intertext{and}
\rond{K}_{s,p}(z)&= \frac{1}{\sin \pi s} \left( \rond{J}_{-s,-p}(z) - \rond{J}_{s,p}(z)\right).
\end{align*}
Because $\rond{J}_{s,p}(z)=\rond{J}_{-s,-p}(z)$ when $s,p\in\Z$, $\rond{K}_{s,p}(z)$ is holomorphic as function of $s$.
Let us define the following Bessel transform:

\begin{defi} 
Let $\rond{H}_\alpha$ be the set of functions defined on $\{s\in\C\,:\, \tq \Re(s)\tq \leqslant \alpha\}\times \Z$ such that

\begin{itemize}
\item[(i)] $h(s,p)=h(-s,-p)$,
\item[(ii)] $h$ is holomorphic on $\tq \Re(s)\tq \leqslant \alpha$,
\item[(iii)] $h(s,p)\ll (1+\tq s\tq)^{-a} (1+\tq p\tq)^{-b}$, for some $a,b >0$.
\end{itemize}
Let $\alpha\in]\frac{1}{3},1[$ and let $h\in \rond{H}_\alpha$. Define $\gras{B}h$ on $\C$ by

\begin{equation*}
\gras{B}h(z)= \frac{1}{2\pi i} \sum_{p\in \Z} \int_{(0)}\rond{K}_{s,p}(z) h(s,p) (p^2-s^2) \,ds.
\end{equation*}
This converges absolutely for $a>2$ and $b>3$.
\end{defi}

The {\it spectral sum formula} is a statement independent of the $K$-type $(l,q)$, that we had to carry until now. The proof is made in two steps, the first one being the computation of the inner product of Poincar\'e series in two ways. One gets a spectral formula depending on the K-type. The second step is then to get rid of the K-type by using an extension method. This work was done by \cite{bru-mot:fourth} for the case ${\rm SL}_2(\Z[i])$ and was extended by \cite{gul:these} for any quadratic number field.

\begin{thm}\label{spectrum:thm:1}
Let $h$ be a function of $\rond{H}_\alpha$, $\alpha \in ]\frac{1}{3},1[$. Let $\Gamma$ be a discrete subgroup of $SL_2(R)$ and let $\sigma^{-1}(\infty)$ and $\tau^{-1}(\infty)$ two essential cusps of $\Gamma$. Then

\begin{gather*}
\sum_{c\neq 0} \frac{K_{\sigma,\tau}(m,n,c)}{\rond{N}(c)} \gras{B}h \left(\frac{4\pi \sqrt{mn}}{c}\right) + \delta_{\sigma,\tau}\delta_{m,n}\frac{1}{2\pi^3i} \sum_{p\in \Z} \int_{(0)} h(s,p) (p^2-s^2)\,ds \\
= \sum_V \barre{\rho_V(m)} \rho_V(n) h(s_V,p_V) \\
 + \frac{1}{2i\pi}\sum_{\sigma_i\in\rond{C}(\Gamma)}\frac{1}{\tq \Lambda_{\sigma_i}\tq } \sum_{p\in \Z} \int_{(0)}  \barre{\psi_{\sigma,\sigma_i}(\sigma,m,p)} \psi_{\sigma_i,\tau}(\tau,n,p) h(s,p)\,ds.
\end{gather*}
\end{thm}
Since Fourier coefficients of representations are the same as the Fourier coefficients of any elements in the space of representations, and taking into account that the multiplicity of $V$ in $L^2\left(\Gamma\backslash G, \kappa\right)$ is the dimension of the space $L^2\left(\Gamma\backslash G,\kappa,(s,p)\right)$, we introduce the following notation corresponding to either the discrete or continuous spectrum of metaplectic forms with respect to the group $\Gamma=\Gamma_d$ (defined in \eqref{Kloo:eq:7}):

\begin{align}
A_{\sigma,\tau,m,n}^{\ec{disc}}(d,s,p) 
&= \sum_{\substack{f\in \ec{ orth. basis of}\\ L^2\left(\Gamma\backslash G,\kappa,(s,p)\right)}} \barre{\rho_f(\sigma,m)} \rho_f(\tau,n), \\
A_{\sigma,\tau,m,n}^{\ec{cont}}(d,s,p)
&= \sum_{\sigma_i} \barre{\psi_{\sigma,\sigma_i}(\sigma,m,p)} \psi_{\sigma_i,\tau}(\tau,n,p),
\end{align}
where the sum over $f$ is taken over an orthonormal basis of the proper subspace of $L^2\left(\Gamma\backslash {\rm SL}_2(\C),\kappa\right)$ corresponding to the spectral parameter $(s,p)$. The sum over the $\sigma_i$'s, means the sum over all essential cusps of the group $\Gamma$. 
With these notations, we rewrite Theorem\,\ref{spectrum:thm:1} as

\begin{thm}\label{spectrum:thm:2}
Let $\alpha\in]\frac{1}{3},1[$ and let $h\in \rond{H}_\alpha$, with  $a>2$ and $b>3$. Let $d$ be a primary Eisenstein integer and let $\sigma^{-1}(\infty)$ and $\tau^{-1}(\infty)$ be two essential cusps of $\Gamma_d$. Let $m,n\in\Z[\omega]-\{0\}$. Then,

\begin{gather*}
\sum_{c\neq 0} \frac{K_{\sigma,\tau}(m,n,c)}{\rond{N}(c)} \gras{B}h\left(\frac{4\pi\sqrt{mn}}{c}\right) 
+\delta_{\sigma,\tau}\delta_{m,n} \sum_{p\in \Z} \int_{(0)} h(s,p) (p^2-s^2)\,ds \\
= \sum_{(s,p)} A_{\sigma,\tau,m,n}^{\ec{disc}}(d,s,p) h(s,p) 
 + \frac{1}{2i\pi}  \sum_{p\in\Z}\int_{(0)} A_{\sigma,\tau,m,n}^{\ec{cont}}(d,s,p) h(s,p)\,ds,
\end{gather*}
where the first sum in the right side is taken over the spectral parameters $(s,p)$.
\end{thm}

One obtains the {\it Kloosterman sum formula} by inverting the Bessel transform $\gras{B}$ on one side.
Let $\gras{K}$ be the Bessel transform be defined by

\begin{equation*}
\gras{K}f (s,p) =\int_{\C^\times} \rond{K}_{s,p}(u) f(u) \tq u\tq^{-2} \,du.
\end{equation*}
Then, any compactly supported function $f$ on $\C^\times$ such that
$\gras{K}f\in \rond{H}_{\alpha}$ for some $\alpha >1$ satisfies the {\it Kloosterman sum formula}:

\begin{thm}\label{spectrum:thm:3}
Let $d$ be a primary Eisenstein integer and let $\sigma^{-1}(\infty)$ and $\tau^{-1}(\infty)$ be two essential cusps of $\Gamma_d$. Let $m,n\in\Z[\omega]-\{0\}$. Let $f\in C^\infty_c(\C^\times)$. Then,

\begin{align*}
& 
\sum_{c} \frac{K_{\sigma,\tau}(m,n,c)}{\rond{N}(c)}  f\left(\frac{4\pi\sqrt{mn}}{c}\right) = 
\sum_{(p,s)} A_{m,n,\sigma,\tau}^{\ec{disc}} (d,s,p) \gras{K}f (s,p) \\
&+ \frac{1}{2i\pi}\sum_{\sigma_i\in\rond{C}(\Gamma)} \sum_{p\in \Z} \int_{(0)}A_{m,n,\sigma,\tau}^{\ec{cont}}(d,s,p) \gras{K}f(s,p) \,ds. 
\end{align*}
\end{thm}

\begin{proof}
The proof amounts to show that

\begin{equation*}
2\pi \gras{B}\gras{K}f=f.
\end{equation*}
As a consequence, substituting $h$ by $\gras{K}f$ in Theorem\,\ref{spectrum:thm:2} gives the result. All details are given in \cite{gul:these}.
\end{proof}

In the rest of this section, we derive from the spectral sum formula some consequences on the spectrum of $L^2\left(\Gamma_d\backslash {\rm SL}_2(\C),\kappa\right)$. This is done by choosing a suitable function $h$ in Theorem\,\ref{spectrum:thm:2}, and by estimating the $\delta$-term and the Kloosterman term; the $\delta$-term will be evaluated directly, and the Kloosterman term will be estimated with the Weil upper bound. 
Let $m$ and $n$ be some fixed integers in $\Z[\omega] - \{0\}$. 

\begin{prop}\label{spectrum:prop}
Let $d\in{\rm SL}_2\left(\Z[\omega]\right)$ be a primary integer. Let $\sigma^{-1}$ be one of the two matrices defined in \eqref{Kloo:eq:8}.  Then,

(i) Let $a> 2$ and $b> 2$. Then, for $X\ll 1$, 
\begin{gather*}
\sum_{\substack{(s,p)\\s\in i\R}} A_{n,n,\sigma,\sigma}^{\ec{disc}}(d,s,p) (1+\tq s\tq)^{-a}  (1+\tq p\tq)^{-b} \\
+ \sum_p\int_{(0)} A_{n,n,\sigma,\sigma}^{\ec{cont}}(d,s,p) (1+\tq s\tq)^{-a}  (1+\tq p\tq)^{-b} \,ds \ll  1.
\end{gather*}

(ii) Let $S$ be a subset of the exceptional spectrum of $\Delta$ in $L^2\left(\Gamma\backslash G, \kappa\right)$. Then, for $x\geqslant 1/2$,

\begin{align*}
\sum_{s_j\in S} A_{n,n,\sigma,\sigma}^{\ec{disc}}(d,s_j,p) \rond{N}(d)^{4xs_j} & \ll  \rond{N}(d)^{2x-1} \tau(d) \log^2(\rond{N}(d)).
\end{align*}
\end{prop}

\begin{proof}
For $(i)$, we apply Theorem\,\ref{spectrum:thm:2} with the special choice 

\begin{equation*}
h(s,p)= (1+\tq s\tq)^{-a} (1+\tq p\tq)^{-b},
\end{equation*}
which gives

\begin{equation}\label{spectrum:eq:prop1}
\begin{split}
&\sum_{(s,p)} A^{\ec{disc}}_{m,m,\sigma,\sigma}(d,s,p) (1+\tq s\tq)^{-a} (1+\tq p\tq)^{-b}  \\
& +\frac{1}{2i\pi} \sum_{p\in\Z}(1+\tq p\tq)^{-b} \int_{(0)} A^{\ec{cont}}_{m,m,\sigma,\sigma}(d,s,p) (1+\tq s\tq)^{-a}\,ds \\
&= \sum_{p\in \Z} \int_{(0)} h(s,p) (p^2-s^2)\,ds   
+ \sum_{c\neq 0} \frac{K_{\sigma,\sigma}(m,m,c)}{\rond{N}(c)} \gras{B}h\left(\frac{4\pi m}{c}\right) .
\end{split}
\end{equation}
For the first integral of the right hand side of \eqref{spectrum:eq:prop1}, we have

\begin{align}
&\sum_{p\in \Z} \int_{(0)} h(s,p)  (p^2-s^2) \,ds \nonumber\\
&= \sum_{p\in \Z} (1+\tq p\tq)^{-b} \int_0^\infty \frac{p^2+t^2}{(1+t)^a} \,dt \nonumber\\
&= \rond{O}_{a,b}(1), \quad \ec{ if $a\geqslant 4$ and $b> 3$}. \label{spectrum:eq:prop2}
\end{align}
Before evaluating the Kloosterman term,
we begin by the transform $\gras{B}h(z)$. Recall that, by definition of $\gras{B}h$ and $\rond{K}_{s,p}(z)$,

\begin{align*}
\gras{B}h(z)
&= \frac{1}{2\pi i} \sum_{p\in \Z} \int_{(0)}\rond{K}_{s,p}(z)  h(s,p) (p^2-s^2) \,ds\\
&= \frac{1}{\pi i} \sum_{p\in \Z} \int_{(0)}\rond{J}_{s,p}(z) h(s,p) \frac{(s^2-p^2)}{\sin \pi s} \,ds.
\end{align*}
Since there is a pole at $0$ for $p\neq 0$, by shifting the integral we obtain, for any $1/2 < \sigma <1$,

\begin{equation}\label{spectrum:eq:prop3}
\gras{B}h(z) 
= \frac{2}{\pi i} \sum_{p\geqslant 0} \int_{(\sigma)} \rond{J}_{s,p}(z) h(s,p) \frac{(s^2-p^2)}{\sin \pi s} \,ds + \frac{2}{\pi i}\sum_{p\geqslant 1} p^2 \rond{J}_{0,p}(z) h(0,p). 
\end{equation}
Here, we have used the fact that $h(s,p)=h(s,-p)$.
We are working with $z$ belonging to some compact set, in which case the estimate $J_s(z)\ll \frac{1}{\Gamma(s+1)} \left(\frac{\tq z\tq}{2}\right)^{\Re(s)}$ is valid. It follows that 

\begin{gather*}
p^2 \rond{J}_{0,p}(z) (1+ p)^{-b} \ll p^2 (1+p)^{-b} \frac{\left(\tq z\tq/2\right)^{2 p}}{(p !)^2},
\end{gather*}
and, by Stirling's formula, we obtain

\begin{equation}\label{spectrum:eq:prop4}
\frac{2}{\pi i}\sum_{p\geqslant 1} p^2 \rond{J}_{0,p}(z) h(0,p)
\ll \sum_{p\neq 0} p (1+ p)^{-b} \left(\frac{\tq z\tq e}{2 p}\right)^{2p}.
\end{equation}
For the integral over $(\sigma)$ in \eqref{spectrum:eq:prop3}, the same estimate as above for $J_s(z)$ leads to

\begin{align*}
\rond{J}_{s,p}(z) 
&\ll \left\vert z/2\right\vert^{2\Re(s)} \Gamma(s+p+1)^{-1} \Gamma(s-p+1)^{-1}\\
&\ll \left\vert z/2\right\vert^{2\Re(s)} \Gamma(s+p+1)^{-1} (-\pi)^{-1} \sin (\pi (s-p)) \Gamma(p-s),
\end{align*}
thus we obtain

\begin{align}
\frac{J_{s,p}(z)}{\sin \pi s} 
&\ll  \left\vert z/2\right\vert^{2\Re(s)} \frac{\tq \Gamma(p-s)\tq}{\tq \Gamma(s+p+1)\tq} \nonumber\\
&\ll \left\vert z/2\right\vert^{2\Re(s)} \frac{\tq \Gamma(-s)\tq}{\tq \Gamma(s)\tq \tq s+p\tq} \prod_{i=0}^{p-1} \frac{\tq j-s\tq}{\tq j+s\tq} \nonumber\\
&\ll \left\vert z/2\right\vert^{2\Re(s)} \frac{\left(\tq s\tq /e\right)^{\Re(-s)}}{\tq s+p\tq \left(\tq s\tq /e\right)^{\Re(s)}} \nonumber\\
&\ll \frac{\left\vert z/ s\right\vert^{2\Re(s)}}{\tq s+p\tq}.\label{spectrum:eq:prop5}
\end{align}
Inserting \eqref{spectrum:eq:prop5} in the integral in \eqref{spectrum:eq:prop3} gives

\begin{align}
\int_{(\sigma)}\rond{J}_{s,p}(z) h(s,p) \frac{(s^2-p^2)}{\sin \pi s} \,ds 
\ll \int_{(\sigma)} \frac{\tq z\tq^{2\sigma}}{\tq s\tq^{2\sigma}} (1+\tq s\tq)^{-a} (1+p)^{-b} \frac{\tq s^2-p^2\tq}{\tq s+p\tq} \,ds \nonumber\\
\ll (1+p)^{-b} \tq z\tq^{2\sigma} \int_{(\sigma)} (1+\tq s\tq)^{-a} \frac{(\tq s\tq+p)}{\tq s\tq^{2\sigma}} \,ds.\label{spectrum:eq:prop6}
\end{align}
Since we assumed $1-2\sigma <0$, we have $(\tq s\tq+p) \tq s\tq^{-2\sigma} \leqslant \sigma^{-2\sigma}(\sigma+p)$ and thus the remaining integral converges, because $a>1$. Finally, combining the estimate\eqref{spectrum:eq:prop6} with \eqref{spectrum:eq:prop4}, we obtain 

\begin{align}
\gras{B}h(z)
& \ll \sum_{p\geqslant 0} (1+p)^{-b} \left(\frac{\tq z\tq}{\sigma}\right)^{2\sigma} (\sigma+p) 
+ \sum_{p\geqslant 1} p (1+ p)^{-b} \left(\frac{\tq z\tq e}{2 p}\right)^{2p}\nonumber\\
& \ll \tq z\tq^{2\sigma} \sum_{p\geqslant 0} (1+p)^{1-b} + \tq z\tq^2 \sum_{p\geqslant 1} (1+p)^{1-b} \left(\frac{\tq z\tq e}{2 p}\right)^{2(p-1)}\nonumber\\
&\ll \rond{N}(z), \quad \ec{for $b\geqslant 3$}.\label{spectrum:eq:prop7}
\end{align}
The Kloosterman term can now be estimated, using Weil's upper bound. From Lemma\,\ref{Kloo:lemme} part $(ii)$, the $c$'s have to be divisible by $d$. Then, from \eqref{Kloo:eq:Weil} and \eqref{spectrum:eq:prop7}, it follows that

\begin{align}
&\sum_{c\congru 0 \pmod d} \frac{K_{\sigma,\sigma}(m,m,c)}{\rond{N}(c)} \gras{B}h\left(\frac{4\pi m}{c}\right)
\ll \sum_{c\congru 0\pmod d} \rond{N}(c)^{-1/2+\epsilon} \gras{B}h\left(\frac{4\pi m}{c}\right)\nonumber\\
& \ll \sum_{c\congru 0\pmod d} \rond{N}(c)^{-1/2+\epsilon} \rond{N}(m) \rond{N}(c)^{-1}\nonumber\\
&\ll \rond{N}(m) \rond{N}(d)^{-3/2+\epsilon}. \label{spectrum:eq:prop8}
\end{align}
Assertion $(i)$ follows from \eqref{spectrum:eq:prop1}, \eqref{spectrum:eq:prop2} and \eqref{spectrum:eq:prop8}. 

For $(ii)$, we use again Theorem\,\ref{spectrum:thm:2} and choose, for some $L>0$ that will be chosen later, 

\begin{equation*}
h(s,p) =\begin{cases} 
\left(\frac{\sin (-iL s)}{Ls}\right)^4 & p=0 \\ 0& p\neq 0.\end{cases}
\end{equation*}
As in the proof of $(i)$, we obtain an upper estimate, dealing with the two terms separately:

Firstly, the contribution of the $\delta$-term is

\begin{equation}\label{spectrum:eq:prop9}
\int_{(0)} h(s,0) s^2\,ds
= \int_0^\infty \left(\frac{\sin (L t)}{iLt}\right)^4 t^2 \,dt\\
=L^{-4} \int_0^\infty \sin^4 (L t) t^{-2} \,dt
=O\left(L^{-3}\right).
\end{equation}

Then, for the Kloosterman term, we note that since $p=0$, there is no residue of $\rond{J}_{s,p}(z)$, thus

\begin{equation*}
\int_{(\sigma)}\rond{J}_{s,0}(z) h(s,0) \frac{(s^2)}{\sin \pi s} \,ds 
\ll \tq z\tq^{2\sigma} \sigma^{1-2\sigma} \int_{(\sigma)}  h(s,0) \,ds,
\end{equation*}
and we obtain 

\begin{equation}\label{spectrum:eq:prop10}
\gras{B}h(z) \ll 
 \tq z\tq^{2\sigma} \sigma^{1-2\sigma} \int_0^\infty \left(\frac{\sin (-iL (\sigma +it))}{L(\sigma+it)}\right)^4 \,dt.
\end{equation}
Moreover, using the upper bound

\begin{gather*}
\tq \sin (-iL (\sigma +it)) \tq= \tq \frac{1}{2i} \left( e^{L(\sigma+it)}-e^{-L(\sigma+it)}\right) \tq \le e^{L\sigma} + e^{-L\sigma}\ll e^{L\sigma},
\end{gather*}
we obtain for the integral in \eqref{spectrum:eq:prop10}

\begin{equation}\label{spectrum:eq:prop11}
\left\vert \int_0^\infty \left(\frac{\sin (-iL (\sigma +it))}{L(\sigma+it)}\right)^4 \,dt \right\vert 
\ll \int_0^\infty \frac{e^{4L\sigma}}{L^4 (\sigma^2+t^2)^2} \,dt
\ll \frac{e^{4L\sigma}}{L^4 \sigma^3}.
\end{equation}
Thus, as in the proof of $(i)$, by Theorem\,\ref{spectrum:thm:2}, \eqref{spectrum:eq:prop9}, \eqref{spectrum:eq:prop11} and the Weil upper bound, we deduce that

\begin{gather*}
\sum_{(s,p)} A^{\ec{disc}}_{m,m,\sigma,\sigma}(d,s,p) h(s,p)\\
\ll L^{-3} + \sum_{\substack{c\congru 0 \pmod d\\ c\neq 0}} \frac{K_{\sigma,\tau}(m,n,c)}{\rond{N}(c)} \rond{N}(m)^\sigma \rond{N}(c)^{-\sigma} \sigma^{1-2\sigma} \frac{e^{4L\sigma}}{L^4 \sigma^3}\\
\ll L^{-3} + \frac{e^{4L\sigma}}{L^4} \rond{N}(m)^\sigma \sum_{\substack{c\congru 0 \pmod d\\ c\neq 0}}\tau(c) \rond{N}(c)^{-1/2 -\sigma}.
\end{gather*}
We estimate the last sum by

\begin{gather*}
\sum_{\substack{c\congru 0 \pmod d\\c\neq 0}}\tau(c) \rond{N}(c)^{-1/2 -\sigma} 
\ll \tau(d) \rond{N}(d)^{-1/2-\sigma} \sum_{c}\tau(c) \rond{N}(c)^{-1/2 -\sigma} \\
=\tau(d) \rond{N}(d)^{-1/2-\sigma} \zeta_{\Q(\omega)}^2\left(\frac{1}{2}+\sigma\right)
\ll \tau(d) \rond{N}(d)^{-1/2-\sigma} \left(\sigma-\frac{1}{2}\right)^{-2},
\end{gather*}
and we finally get

\begin{equation}\label{spectrum:eq:prop12}
\sum_{(s,p)} A^{\ec{disc}}_{m,m,\sigma,\sigma}(d,s,p)  h(s,p)\\
\ll L^{-3} + \frac{e^{4L\sigma}}{L^4} \rond{N}(m)^\sigma \tau(d) \rond{N}(d)^{-1/2-\sigma} \left(\sigma-\frac{1}{2}\right)^{-2}.
\end{equation}
We look now at a lower estimate: this is simply 

\begin{equation}\label{spectrum:eq:prop13}
 \sum_{(s,p)} A^{\ec{disc}}_{m,m,\sigma,\sigma}(d,s,p) h(s,p) \gg L^{-4}  \sum_{s_j} A^{\ec{disc}}_{m,m,\sigma,\sigma}(d,s,p) e^{4Ls_j} .
\end{equation}
Bringing together the lower estimate \eqref{spectrum:eq:prop13} and the upper estimate \eqref{spectrum:eq:prop12}, we obtain

\begin{equation}\label{spectrum:eq:prop14}
\sum_{s_j} A^{\ec{disc}}_{m,m,\sigma,\sigma}(d,s,p) e^{4Ls_j}
\ll   L + e^{4L\sigma} \rond{N}(m)^\sigma \tau(d) \rond{N}(d)^{-1/2 -\sigma} (\sigma -\frac{1}{2})^{-2}.
\end{equation}
Choose now $\sigma=\frac{1}{2}+ \log^{-1}(\rond{N}(d))$ and $L=4+x\log(\rond{N}(d))$, with $x\geqslant 1/2$; this gives

\begin{equation*}
\sum_{s_j\in ]0,1/3[} A^{\ec{disc}}_{m,m,\sigma,\sigma}(d,s,p)  \rond{N}(d)^{4xs_j}
\ll_{m,x} \rond{N}(d)^{2x-1}  \tau(d) \log^2(\rond{N}(d)).
\end{equation*}
This proves assertion $(ii)$ of Proposition\,\ref{spectrum:prop}.
\end{proof}

\section{Asymptotic behavior over arithmetic progressions}\label{sec:proof}

In this section we prove Theorem\,\ref{intro:thm:2} and its corollaries.  
Let $g$ be a smooth function with compact support included in $[1,2]$. 
Let $f:\C\too \R$ be the smooth function defined by

\begin{equation}\label{proof:eq:1}
f(z)= g\left(\frac{\rond{N}(mn)^{1/2}}{\rond{N}(z) X}\right) \rond{N}(z)^{-1/2} \rond{N}(mn)^{1/2}.
\end{equation}
Then $f$ is a radial function with compact support in $\left[ \rond{N}(mn)^{1/2}/(2X), \rond{N}(mn)^{1/2}/X\right]$ and with $\norme f\norme_\infty \ll_{m,n} \norme g\norme_\infty X^{1/2}$. 
For our purpose we shall need an estimation for $\gras{K}f(s,p)$:

\begin{lemme}\label{proof:lemme}
Let $f$ be the function defined in \eqref{proof:eq:1}. Then

\begin{align*}
(i)&&
\gras{K}f (s,p) & \ll X^{1/2} \frac{1}{(1+\tq p\tq)^a} \frac{1}{(1+\tq s\tq)^b}, \quad \forall a,b \geqslant 1 \ec{ and } s\in i\R,\\
(ii)&&
\gras{K}f(s,0) & = c_s X^{1/2+s} \hat{g}(1/2-s) + O(X^{1/2-s}), \quad \forall s\in ]0,1/2],
\end{align*}
where $c_s$ is a constant depending only on $m$, $n$ and $s$.
\end{lemme}

\begin{proof} 
Recall that $\gras{K}f (s,p)$ is defined by 

\begin{align*}
\gras{K}f (s,p) =\int_\C \rond{K}_{s,p}(z) f(z) \tq z\tq^{-2} \,dz
\end{align*}
and that $\rond{K}_{s,p}(z)= \frac{1}{\sin \pi s} \left( \rond{J}_{-s,-p}(z) - \rond{J}_{s,p}(z)\right)$ is holomorphic at $s=0$. We define 

\begin{align*}
\gras{J}f (s,p) =\int_\C \rond{J}_{s,p}(z) f(z) \tq z\tq^{-2} \,dz
\end{align*}
and start with the expression of $\rond{J}_{s,p}(z)$ as entire series: 

\begin{align*}
\left\vert \frac{z}{2}\right\vert^{-2s} \left( \frac{z}{\tq z\tq}\right)^{2p} \rond{J}_{s,p}(z) 
&= \left\vert \frac{z}{2}\right\vert^{-2s} \left( \frac{z}{\tq z\tq}\right)^{2p} J_{s-p}(z) J_{s+p}(z)\\
&= \sum_{m,n\geqslant 0} \frac{(-z^2/4)^m (-\barre{z}^2/4)^n}{m! n! \Gamma(s-p+m+1) \Gamma(s+p+n+1)}. 
\end{align*}
Then 

\begin{align*}
&\gras{J}f(s,p) =\int_{\C^\times} f(z) \rond{J}_{s,p}(z) \frac{dz}{\tq z\tq^2}\\
&= \sum_{m,n\geqslant 0} \frac{\int_{\C^\times}f(z) \tq \frac{z}{2}\tq^{2s} \left( \frac{iz}{\tq z\tq}\right)^{-2p} (-z^2/4)^m (-\barre{z}^2/4)^n \tq z\tq^{-2} dz}{m! n! \Gamma(s-p+m+1) \Gamma(s+p+n+1)}\\
&= \sum_{m,n\geqslant 0} \frac{4^{-(s+m+n)} (-1)^{p+m+n} \int_{\C^\times} f(z) \tq z\tq^{2(s+m+n)}\left(z/\tq z\tq\right)^{-2(p-m+n)}\tq z\tq^{-2} dz}{m! n! \Gamma(s-p+m+1) \Gamma(s+p+n+1)}\\
&= \sum_{m,n\geqslant 0} \frac{4^{-(s+m+n)} (-1)^{p+m+n} \hat{\hat{f}}(s+m+n,p-m+n)}{m! n! \Gamma(s-p+m+1) \Gamma(s+p+n+1)},
\end{align*}
where $\hat{\hat{f}}$ is the complex Mellin transform of $f$, defined by 

\begin{equation*}
\hat{\hat{f}}(s,p) =\int_{\C^\times} f(z) \tq z\tq^{2s} (z/\tq z\tq)^{-2p} \tq z\tq^{-2} dz.
\end{equation*}
The complex Mellin transform $\hat{\hat{f}}$ is related to the Mellin transform $\hat{f}$ by $\hat{\hat{f}} (s,p) =2\pi \hat{f_p} (2s)$, where $f_p(r)=\frac{1}{2\pi}\int_0^{2\pi} f(re^{i\theta}) e^{-2pi\theta} d\theta$. Now, since $f$ is radial, we have $f_p(r)=f(r)$ if $p=0$, and $f_p(r)=0$ otherwise. Therefore,

\begin{align*}
\gras{J}f(s,p) 
=2\pi \sum_{j\geqslant 0} \frac{4^{-(s+p+2j)} \hat{f}(2(s+p+2j))}{(j+p)! j! \Gamma(s+j+1) \Gamma(s+p+j+1) }.
\end{align*}
From the definition of $f$ it follows that

\begin{align*}
&\hat{f}(2(s+p+2j))
= \int_0^\infty g\left(\frac{\sqrt{\rond{N}(m)\rond{N}(n)}}{r^2 X}\right) \sqrt{\rond{N}(m)\rond{N}(n)} r^{2(s+p+2j)-2} \,dr\\
&= \frac{1}{2} X^{1/2-s-p-2j} (\rond{N}(m)\rond{N}(n))^{(s+p)/2+j+1/4} \int_0^\infty g(t) t^{1/2- (s+p+2j)} \,\frac{dt}{t}.
\end{align*}

To prove $(i)$, let us assume that $t=\Im(s)>1$. Then

\begin{align*}
&\gras{J}f(s,p) =\pi 4^{-s} X^{1/2-s} (\rond{N}(m)\rond{N}(n))^{s/2+1/4}\\
& \sum_{j\geqslant 0} \frac{4^{-(p+2j)} X^{-p-2j} (\rond{N}(m)\rond{N}(n))^{p/2+j} \hat{g}(1/2-s-p-2j)}{(j+p)! j! \Gamma(s+j+1) \Gamma(s+p+j+1)}.
\end{align*}
We use the estimate 

\begin{align*}
\hat{g}(s) 
&=\int_0^\infty g(r) r^{s-1} dr \\
&= (-1)^a \int_0^\infty f^{(a)}(r) \frac{r^{s-1+a}}{s\pts (s-1+a)} \,dr \\
&\ll \left(1+\tq s\tq\right)^{-a}
\end{align*}
and obtain

\begin{align*}
\gras{J}f(s,p) 
&\ll \sum_{j\geqslant 0} \frac{4^{-(p+2j)} X^{-p-2j} (\rond{N}(m)\rond{N}(n))^{p/2+j}}{p! j!(1+\tq t\tq)^a \Gamma(s+j+1) \Gamma(s+p+j+1)}.
\end{align*}
Now, writing $(s)_k$ for the quantity $s (s+1) \pts (s+k-1)$, we estimate the product of $\Gamma$-factors as follows:

\begin{gather*}
\Gamma(s+j+1) \gamma(s+p+j+1)=(s)_{j+1} (s)_{p+j+1} \Gamma(s)^2\\
= (s)_{j+1} (s)_{p+j+1} \frac{\Gamma(s)}{s \Gamma(-s) \sin \pi s} \gg \tq s\tq (1+\tq t\tq)^{2j+p} e^{-\pi \tq t\tq}.
\end{gather*}
Therefore, 

\begin{align*}
\gras{J}f(s,p) 
&\ll \sum_{j\geqslant 0} \frac{4^{-(p+2j)} X^{-p-2j} (\rond{N}(m)\rond{N}(n))^{p/2+j} e^{\pi \tq t\tq}}{p! j! (1+\tq t\tq)^a (1+\tq t\tq)^{2j+p}}\\
&\ll \frac{e^{\pi \tq t\tq}}{ p! (1+\tq t\tq)^a} \ll \frac{e^{\pi \tq t\tq}}{(1+\tq p\tq)^b (1+\tq t\tq)^a}.
\end{align*}
We use the definition of $\gras{K}$ in terms of $\gras{J}$ to get the final estimate. To conclude the proof of $(i)$, it remains to treat the case where $\Im(s)\leqslant 1$; this follows from the previous case, by representing $\gras{K}f(s,p)$ as an integral on a circle on which the estimate holds already.

To prove $(ii)$, we isolate the term corresponding to $n=0$. Then by integration by part, and because $p=0$, it remains 

\begin{align*}
\gras{J}f(s,0) 
&= \pi X^{1/2-s}\frac{4^{-s}}{\Gamma(s+1)^2} (\rond{N}(m)\rond{N}(n))^{s/2+1/4} \hat{g}(1/2-s)\\
&+ \pi X^{1/2-s} (\rond{N}(m)\rond{N}(n))^{s/2+ 1/4}\\
& \quad \sum_{n\geqslant 1} \frac{4^{-(s+2n)} X^{-2n} \int_0^\infty g^{(2n)}(t) t^{1/2-s} dt/t}{(n!)^2(\Gamma(s+n+1))^2 (1/2-s+1) \pts (1/2-s+2n)}\\
&= \pi X^{1/2-s}\frac{4^{-s}}{\Gamma(s+1)^2} (\rond{N}(m)\rond{N}(n))^{s/2+1/4} \hat{g}(1/2-s) +\rond{O}(X^{-3/2}).
\end{align*}
Replacing $s$ by $-s$, we obtain the corresponding result for $\gras{J}f(-s,0)$. It remains to subtract and to divide by $\sin \pi s$; we remark that this last operation does not produce any pole. It follows that

\begin{align*}
\gras{K}f (s,0) 
&=c_s X^{1/2+s}\hat{g}(\frac{1}{2}+s) +\rond{O}(X^{-1})+ c_{-s} X^{1/2-s} \hat{g}(\frac{1}{2}-s) +\rond{O}(X^{-3/2})\\
&= c_s X^{1/2+s}\hat{g}(\frac{1}{2}+s) + \rond{O}(X^{1/2-s}),
\end{align*}
as announced, where the constant $c_s$ is given by

\begin{equation*}c_s=\frac{\pi 4^s}{\Gamma(1-s)^2} \frac{1}{\sin \pi s} (\rond{N}(m)\rond{N}(n))^{(1-2s)/4}.
\end{equation*}
\end{proof}

In order to prove Theorem\,\ref{intro:thm:2}, it is convenient to define, for a given primary Eisenstein integer $d$, and for $0\leqslant \ell \leqslant \frac{1}{3}$, the quantity

\begin{equation}\label{proof:eq:2}
\begin{split}
& S(\ell) = \sum_{\substack{0\leqslant s \leqslant 1/2 \\ s\leqslant \ell}}  \gras{K}f(s,0) A_{m,n,\sigma,\tau}^{\ec{disc}}(d,s,0) \\ 
& + \sum_{\substack{(s,p)\\ s\in i\R}} \gras{K}f(s,p) A_{m,n,\sigma,\tau}^{\ec{disc}}(d,s,p) + \sum_{p\in\Z}\int_{(0)} \gras{K}f(s,p) A_{m,n,\sigma,\tau}^{\ec{cont}}(d,s,p) \,ds.
\end{split}
\end{equation}
We shall be later interested in $S(\ell)$ in the cases $\ell=\frac{1}{3}$ and $\ell =s(d)$, with $s(d)$ being as in Theorem\,\ref{meta:thm:1}. Using Lemma\,\ref{proof:lemme}, but  in the simpler form $Kf(s,0)=O(X^{1/2+s})$ in the case $0<s<1/2$, we get
  
\begin{equation}\label{proof:eq:3}
\begin{split}
\tq S(\ell) \tq  & \leqslant X^{1/2} \Bigg( \sum_{\substack{0<s<1/2 \\ 0<s\leqslant \ell}} X^{s} \left\vert A_{m,n,\sigma,\tau}^{\ec{disc}}(d,s,0) \right\vert \\
&+ \sum_{\substack{(s,p)\\ s\in iR}} \frac{ \left\vert A_{m,n,\sigma,\tau}^{\ec{disc}}(d,s,p)\right\vert}{(1+\tq s\tq)(1+\tq p\tq)} +  \sum_{p\in\Z}\int_{(0)} \frac{\left\vert A_{m,n,\sigma,\tau}^{\ec{cont}}(d,s,p)\right\vert}{(1+\tq s\tq) (1+\tq p\tq)} \,ds\Bigg).
\end{split}
\end{equation}
 
For ${}^*$ representing either the discrete case or the continuous case, we have from the Cauchy-Schwarz inequality: 

\begin{equation}\label{proof:eq:4}
\left\vert A^*_{m,n,\sigma,\tau}(d,s,p) \right\vert \leqslant \left(A^*_{m,m,\sigma,\sigma}(d,s,p)\right)^{1/2} \left(A^*_{n,n,\tau,\tau}(d,s,p)\right)^{1/2}.
\end{equation}
We have to separate the exceptional and non-exceptional spectrum. By \eqref{proof:eq:4} and the Cauchy-Schwarz inequality, the last two sums of \eqref{proof:eq:3} are bounded by

\begin{equation}\label{proof:eq:5}
\begin{split}
&  \left(\sum_{\substack{(s,p) \\ s\in i\R}} \frac{A^{\ec{disc}}_{m,m,\sigma,\sigma}(d,s,p)}{(1+\tq s\tq)^a (1+\tq p\tq)^b} \right)^{1/2} \left(\sum_{\substack{(s,p) \\ s\in i\R}} \frac{ A^{\ec{disc}}_{n,n,\tau,\tau}(d,s,p)}{(1+\tq s\tq)^a (1+\tq p\tq)^b} \right)^{1/2} \\
&+ \left( \sum_{p\in\Z} \int_{(0)} \frac{A^{\ec{cont}}_{m,m,\sigma,\sigma}(d,s,p)}{(1+\tq s\tq)^a (1+\tq p\tq)^b} \,ds \right)^{1/2}
\left( \sum_{p\in\Z} \int_{(0)} \frac{A^{\ec{cont}}_{n,n,\tau,\tau}(d,s,p)}{(1+\tq s\tq)^a (1+\tq p\tq)^b} \,ds \right)^{1/2}\!\!.
\end{split}
\end{equation}
Since for positive quantities $a$ and $b$, one has $\sqrt{ab}+\sqrt{cd} \leqslant \sqrt{a+c} \sqrt{b+d}$, the expression in \eqref{proof:eq:5} is bounded by

\begin{align*}
& \left( \sum_{\substack{(s,p) \\ s\in i\R}} \frac{A^{\ec{disc}}_{m,m,\sigma,\sigma}(d,s,p)}{(1+\tq s\tq)^a (1+\tq p\tq)^b} +  \sum_{p\in\Z} \int_{(0)} \frac{A^{\ec{cont}}_{m,m,\sigma,\sigma}(d,s,p)}{(1+\tq s\tq)^a (1+\tq p\tq)^b} \,ds \right)^{1/2}\\
& \left( \sum_{\substack{(s,p) \\ s\in i\R}} \frac{ A^{\ec{disc}}_{n,n,\tau,\tau}(d,s,p)}{(1+\tq s\tq)^a (1+\tq p\tq)^b} + \sum_{p\in\Z} \int_{(0)} \frac{A^{\ec{cont}}_{n,n,\tau,\tau}(d,s,p)}{(1+\tq s\tq)^a (1+\tq p\tq)^b} \,ds \right)^{1/2}.
\end{align*} 
By part $(i)$ of Proposition\,\ref{spectrum:prop}, we conclude that the last two sums in \eqref{proof:eq:3} are bounded by $\rond{O}(1)$. 
By \eqref{proof:eq:4} and the Cauchy-Schwarz inequality, we obtain, in the same way as above, that the first sum of \eqref{proof:eq:3} is bounded by

\begin{align*}
& \left(\sum_{\substack{0<s<1/2 \\ s\leqslant \ell}}X^s A^{\ec{disc}}_{m,m,\sigma,\sigma}(d,s,0)\right)^{1/2} \
\left(\sum_{\substack{0<s<1/2 \\ s\leqslant \ell}}X^s A^{\ec{disc}}_{n,n,\tau,\tau}(d,s,0)\right)^{1/2}.
\end{align*}
Each of both sums can be estimated  as follows:

\begin{align*}
&\sum_{\substack{0<s<1/2 \\ s\leqslant \ell}} X^s A_{*}^{\ec{disc}}(d,s,0)\\
& \ll \left(\frac{X}{\rond{N}(d)^2}\right)^\ell 
\sum_{\substack{0<s<1/2 \\ s\leqslant \ell}} \rond{N}(d)^{2s} A_{*}^{\ec{disc}}(d,s,0), \quad \ec{whenever $\rond{N}(d)^2 \le X$}\\
&\ll \left(\frac{X}{\rond{N}(d)^2}\right)^\ell \ \tau(d) \log^2 \rond{N}(d), \quad \ec{ by part $(ii)$ of Proposition\,\ref{spectrum:prop}}\\
&\ll X^{\ell} \log^2X \rond{N}(d)^{-2\ell} \tau(d).
\end{align*} 
Therefore the first sum of \eqref{proof:eq:3} is bounded by $X^\ell \log^2 X \tau(d) \rond{N}(d)^{-2\ell}$,
and we have proved that

\begin{equation}\label{proof:eq:6}
S(\ell) = \rond{O}\left( X^{1/2} + X^{1/2+\ell} \log^2 X \frac{\tau(d)}{\rond{N}(d)^{2\ell}}\right).
\end{equation}

We can now come back to our original problem. With our choice of $f$ in \eqref{proof:eq:1}, using \eqref{proof:eq:6} and Lemma\,\ref{proof:lemme} $(ii)$, we have, on the one hand:

\begin{align}
&
\sum_{c} \frac{K_{\sigma,\tau}(m,n,c)}{\rond{N}(c)^{1/2}}  g\left(\frac{\rond{N}(c)}{X}\right)\nonumber\\
&= \sum_{c} \frac{K_{\sigma,\tau}(m,n,c)}{\rond{N}(c)} f\left(\frac{\sqrt{\rond{N}(m)\rond{N}(n)}}{c}\right)\nonumber\\
&= A_{m,n,\sigma,\tau}\left(d,1/3,0\right) \gras{K}f \left(1/3,0\right)\nonumber + 
\sum_{c} \frac{K_{\sigma,\tau}(c)}{\rond{N}(c)} f\left(\frac{\sqrt{\rond{N}(m)\rond{N}(n)}}{c}\right)\\
&\quad -A_{m,n,\sigma,\tau}\left(d,1/3,0\right) \gras{K}f \left(1/3,0\right)\nonumber\\
& = X^{5/6} A_{m,n,\sigma,\tau}\left(d,1/3,0\right) c_{1/3} \hat{g}(1/6) + \rond{O}\left(X^{1/6}\right)+ S\left( s(d)\right) \nonumber\\
&\begin{aligned}
&= X^{5/6} A_{m,n,\sigma,\tau}\left(d,1/3,0\right) c_{1/3} \hat{g}(1/6) \\
&\quad +\rond{O}\left( X^{1/2} + X^{1/2+s(d)} \log^2 X \frac{\tau(d)}{\rond{N}(d)^{2s(d)}}\right).
\end{aligned}\label{proof:eq:7}
\end{align}
One the other hand, we also have 

\begin{align}
&\sum_{c} \frac{K_{\sigma,\tau}(m,n,c)}{\rond{N}(c)^{1/2}}  g\left(\frac{\rond{N}(c)}{X}\right)
= \sum_{c} \frac{K_{\sigma,\tau}(m,n,c)}{\rond{N}(c)} f\left(\frac{\sqrt{\rond{N}(m)\rond{N}(n)}}{c}\right)\nonumber\\
& =  S\left(1/3\right) 
= \rond{O}\left( X^{1/2} + X^{5/6} \log^2 X \frac{\tau(d)}{\rond{N}(d)^{2/3}}\right).\label{proof:eq:8}
\end{align}

Let us choose now $m=n=1$ and $\sigma$ and $\tau$ as in \eqref{Kloo:eq:8}. With this choice, we can apply Lemma\,\ref{Kloo:lemme}. Then, \eqref{proof:eq:7} proves part $(i)$ of Theorem\,\ref{intro:thm:2}, with the constant $c_\theta(d)$  defined by 

\begin{equation*}
c_\theta(d) =A_{1,1,\sigma,\tau} \left(d,1/3,0\right) c_{1/3},
\end{equation*}
where $c_{1/3}$ is defined in part\,$(ii)$ of Lemma\,\ref{proof:lemme}.
Next, \eqref{proof:eq:8} proves part\,$(ii)$ of Theorem\,\ref{intro:thm:2}. This concludes the proof of Theorem\,\ref{intro:thm:2}.

Corollary\,\ref{intro:cor:1} is a consequence of $s(1)=0$. 
If we denote by $\lambda_1$ the smallest eigenvalue of the Laplacian for $SL_2(R)$ other than $8/9$, then $s(1)=0$ means that $\lambda_1\geqslant 1$. Actually, sharper estimates for $\lambda_1$ are available. See for example \cite{str:eigen}, where the lower bound $\lambda_1 > \pi^2 32/27$ is proved.

To prove Corollary\,\ref{intro:cor:2}, it remains to take the sum over the $d$'s such that $\rond{N}(d)\leqslant D$, for $D< X^{1/2}$. We obtain from \eqref{proof:eq:8}:

\begin{equation}\label{proof:eq:9}
\sum_{\rond{N}(d)\leqslant D} \sum_{c} \frac{S(1,c)}{\rond{N}(c)^{1/2}}  g\left(\frac{\rond{N}(c)}{X}\right)
\ll  X^{1/2} D + X^{1/2+s_1} \log^2 X \sum_{\rond{N}(d)\leqslant D} \frac{\tau(d)}{\rond{N}(d)^{2s_1}}.
\end{equation} 
Here $s_1$ is the spectral parameter $1/3$; writing it in this way, we remark that, contrary to the non-metaplectic case where the maximal exceptional spectral parameter $s_1$ could possibly be very close to $1/2$ for some $d$, we do not need here to improve on  Selberg's estimate (compare with \cite[$(2.3)$ p.\,11]{fou-mic:sommes-mod}), since the value $s_1=1/3$ is already strictly smaller than $1/2$. It follows that the right hand side of \eqref{proof:eq:9} is estimated by $\rond{O}\left(X^{1/2}D+X^{5/6}D^{1/3}\right)$. This concludes the proof of Corollary\,\ref{intro:cor:2}. 

\section{A sieve method}\label{sec:sieve}

In this section we deduce the upper bound \eqref{intro:eq:upper} form Corollary\,\ref{intro:cor:2}. Since the method is taken from \cite{fou-mic:chgmt-signe}, we do not need to give all the details here. 
The problem is to apply a sieve method to the sequence of real numbers 
$S(1,c)/\sqrt{\rond{N}(c)} g\left( \frac{\rond{N}(c)}{X}\right)$. A solution has been found in \cite{fou-mic:chgmt-signe}, for the case of Kloosterman sums. For $c\in R$, let $\Omega(c)$ be the number of prime factors of $c$, counted with multiplicity. Recall that $\lambda=1-\omega$ is a prime in $R$ and that $\lambda^2=-3\omega$. Let us define the following three sequences of numbers: $\rond{A}^+=(a^+_c)$, $\rond{A}^-=(a^-_c)$ and $\rond{B}=(b_c)$, where
 
\begin{align*}
a_c^\pm &= \begin{cases}
g\left( \frac{\rond{N}(c)}{X}\right) \left( \frac{\pm S(1,c)}{\sqrt{\rond{N}(c)}} + 2^{\Omega(c)}\right)
& \ec{if $\lambda \nmid c$}\\
0& \ec{if $\lambda \tq c$}
\end{cases}
\intertext{and}
b_c&= \begin{cases}
g\left( \frac{\rond{N}(c)}{X}\right) 2^{\Omega(c)}
& \ec{if $\lambda \nmid c$}\\
0& \ec{if $\lambda \tq c$}.
\end{cases}
\end{align*}
Then by the Weil bound \eqref{Kloo:eq:Weil}, these numbers are non-negative. Let us now consider the quantities $S(\rond{A}^\pm,z)$ and $S(\rond{B},z)$ defined by

\begin{equation*}
S(\rond{A}^\pm,z) =\sum_{\pi \tq c \Rightarrow \rond{N}(\pi)\geqslant z}a^\pm_c \quad \ec{ and } \quad S(\rond{B},z) =\sum_{\pi \tq c \Rightarrow \rond{N}(\pi)\geqslant z}b_c.
\end{equation*}
We will be interested in the case where $z=X^{1/u}$. One shows firstly that 

\begin{prop}\label{sieve:prop:1} 
Let $X \gg 1$ and let $u\geqslant 1$. Then

\begin{equation*}
S(\rond{B},X^{1/u}) =c \hat{g}(1) \frac{X}{\log X} (u^2+2u) \left(1+h_1(u) +\rond{O}\left(\frac{1}{\log X}\right)\right),
\end{equation*}
where $c=e^{-2\gamma}$ and $h_1(u) = \rond{O}\left((u/5)^{-u/5}\right)$. Moreover, if $u\geqslant 22$, one has $\tq h_1(u)\tq<10^{-8}$.
\end{prop}

\begin{proof} 
See \cite{fou-mic:chgmt-signe}, $(5.10)$ p. 24, and \cite{siv:kloo}, Lemme\,3.2.  
\end{proof}

The term $u^2+2u$ appearing in Proposition\,\ref{sieve:prop:1} produces a difficulty here. However, it is possible to sieve the sequence $\rond{A}^\pm$ and to control the dependence on $u$ of the main term. This is shown by the following result:

\begin{prop}\label{sieve:prop:2} 
Let $X \gg 1$ and let $u\geqslant 1$. Then

\begin{equation*}
S(\rond{A}^\pm,X^{1/u}) \leqslant 
c \hat{g}(1) \frac{X}{\log X} (u^2+2u) \left( 1+h_2(u) +\rond{O}_{g,u}\left(\frac{1}{\log X}\right)\right).\end{equation*}
where $c=e^{-2\gamma}$, and $h_2(u)=\rond{O}\left((u/5)^{-u/5}\right)$. Moreover, if $u\geqslant 60$, one has $\tq h_2(u)\tq < 10^{-8}$. 
\end{prop}

Assuming Proposition\,\ref{sieve:prop:2} for a moment, one sees that the estimate \eqref{intro:eq:upper} of the introduction is obtained by subtracting $S(\rond{B},X^{1/u})$ from $S(\rond{A}^\pm,X^{1/u})$. 

The rest of this section is devoted to the proof of Proposition\,\ref{sieve:prop:2}; this is done by sieving the sequence $\rond{A}^\pm$ and applying Corollary\,\ref{intro:cor:2}. 
As usual in sieve theory, one is interested in the quantity 

\begin{equation*}
\rond{A}^\pm_d= \sum_{c\congru 0\pmod d} a^\pm_c.
\end{equation*}
Since $\rond{A}^\pm$ is defined in terms of $\rond{B}$, we start with the local behavior of the sequence $\rond{B}$. Let us define $L(s)=\sum 2^{\Omega(c)} \rond{N}(c)^{-s}$, the sum being taken over the integers $c\in R$ coprime to $3$. This function has a pole at $s=1$. We will need the function 

\begin{equation*}
F_\lambda(s)=\prod_{\pi\neq \lambda} \left( 1 +\frac{1}{\rond{N}(\pi)^s (\rond{N}(\pi)^s-2)}\right),
\end{equation*}
for $\Re(s) >1/2$. By use of the Mellin transform and the residue theorem, one shows that for any $1/2>\varepsilon >0$ and $\sigma>1$,

\begin{align}
\sum_{\lambda \nmid c} b_c
&= \frac{1}{2i\pi} \int_{(\sigma)}\hat{g}(s) L(s) X^s \,ds\nonumber \\
&= X\log X \alpha^2 \frac{4}{9} \hat{g}(1) F_\lambda(1) +X d(g)+ \frac{1}{2i\pi} \int_{(1-\varepsilon)}\hat{g}(s) L(s) X^s \,ds\nonumber \\
&= X\log X \alpha^2 \frac{4}{9} \hat{g}(1) F_\lambda(1) +X d(g)+ \rond{O}\left(X^{1-\varepsilon}\right).\label{sieve:eq:1}
\end{align}
We recall that $\alpha$ is the residue at $s=1$ of the Dedekind zeta function $\zeta_k(s)$ of the field $k=\Q(\omega)$, thus $\alpha=\pi / 3\sqrt{3}$.  In \eqref{sieve:eq:1}, $d(g)$ is a quantity depending on the function $g$, that we shall not need to make precise.
Since $2^{\Omega(c)}$ is totally multiplicative, we deduce from \eqref{sieve:eq:1} that

\begin{align}
\rond{B}_d&= \sum_{\substack{c\congru 0\pmod d\\ \ec{gcd}(c,3)=1}} b_c 
=2^{\Omega(d)} \sum_{\lambda \nmid c} 2^{\Omega(c)} g\left( \frac{\rond{N}(c)}{ X/\rond{N}(d)}\right)\nonumber\\
& \begin{aligned}
&=\frac{2^{\Omega(d)}}{\rond{N}(d)} X \left( \log X \alpha^2 \frac{4}{9} \hat{g}(1) F_\lambda(1)- \log \rond{N}(d) \alpha^2 \frac{4}{9} \hat{g}(1) F_\lambda(1)+d(g)\right) \\
& \quad
+ \rond{O} \left( 2^{\Omega(d)} \left(\frac{X}{\rond{N}(d)}\right)^{1-\varepsilon}\right).
\end{aligned}\label{sieve:eq:2}
\end{align}
We define the real constants $Y$, $Z$ and $r(d)$, depending on $X$ and $g$, by

\begin{equation}\label{sieve:eq:3}
Y= X\log X\alpha^2 \frac{4}{9} \hat{g}(1) F_\lambda(1)+Xd(g), \quad 
Z= X\alpha^2 \frac{4}{9} \hat{g}(1) F_\lambda(1)
\end{equation}
and 

\begin{equation}\label{sieve:eq:4}
r(d) = \sum_{\substack{c\congru 0\pmod d \\ \lambda \nmid c}}
\frac{\pm S(1,c)}{\sqrt{\rond{N}(c)}} g\left(\frac{\rond{N}(c)}{X}\right)+ \rond{O} \left( 2^{\Omega(d)} \left(\frac{X}{\rond{N}(d)}\right)^{1-\varepsilon}\right).
\end{equation}
Let us also introduce the completely multiplicative function $\rho$, defining its value  at an Eisenstein prime $\pi\in R$ by

\begin{equation}\label{sieve:eq:5}
\rho(\pi)= \begin{cases} 2&\ec{if $\pi\neq \lambda$}\\0&\ec{otherwise}.\end{cases}
\end{equation}
With these notations, it follows from \eqref{sieve:eq:2} that $\rond{A}^\pm$ satisfies

\begin{equation}\label{sieve:eq:6}
\rond{A}^\pm_d = \frac{\rho(d)}{\rond{N}(d)} Y
- \frac{\rho(d)}{\rond{N}(d)} \log \rond{N}(d) Z + r(d).
\end{equation}  
This is an unusual sieve hypothesis, but can nevertheless be handled. We use the  {\it Crible \'etrange}, developed by Fouvry and Michel in \cite{fou-mic:chgmt-signe}. Some more notations are required: let $\rho^*(\pi)=\rond{N}(\pi)-\rho(\pi)$, and let $\rho^*$ be defined as a completely multiplicative function. Let $g$ be the function

\begin{equation*}
g(d) =\begin{cases}
\rho(d) \rho^*(d)^{-1} &\ec{if $d\in R$ is square-free}\\
0&\ec{otherwise}.\end{cases}
\end{equation*}
Let us define, for $z>0$, the sum

\begin{equation*}
G_c(T,z)= \sum_{\substack{\rond{N}(d)\leqslant T \\ d\tq \rond{P}(z) \\ \ec{gcd}(d,c)=1}} g(d),
\end{equation*}
where $\rond{P}(z)$ is the product over the primes of norm less than or equal to $z$. We write $G(T,z)$ for $G_1(T,z)$. We will also need the function $\sigma_2$ defined on $\R$ by

\begin{align}
\sigma_2(u) &= \frac{e^{-2\gamma}}{8} u^2, & 
& 0\leqslant u \leqslant 2 \label{sieve:eq:7}\\
\left(u^{-2} \sigma_2(u)\right)' &= -2u u^{-3} \sigma_2(u-2),&
2<u, \label{sieve:eq:8}
\end{align}
with $\sigma_2$ required to be continuous at $u=2$. It is a non-negative increasing function with $\lim_{u\to \infty}\sigma_2(u)=1$. 

\begin{thm}\label{sieve:thm} 
Let $\rond{A}$ be a sequence of positive numbers. Assume that there exist $Y$, $Z$ and $r(d)$ such that $\rond{A}$ satisfies \eqref{sieve:eq:6}, where $\rho$ is the totally multiplicative function defined in \eqref{sieve:eq:5}. Then, with the notations introduced above, the following inequality holds, for any $D\geqslant 1$:

\begin{equation*}
\begin{split}
S(\rond{A},z) \leqslant & Y G(D,z)^{-1} \\
&+ Z G(D,z)^{-2} \sum_{\pi\tq \rond{P}(z)} \frac{\rho(\pi) \rond{N}(\pi)}{\rho^*(\pi)^2} \log \rond{N}(\pi) G_\pi(D/\rond{N}(\pi), z)\\
&+ \sum_{\substack{d\tq \rond{P}(z) \\ \rond{N}(d) \leqslant D^2}} 3^{\Omega(d)} \tq r(d)\tq.
\end{split}
\end{equation*}
Moreover, if $z\leqslant D$, 

\begin{align*}
G_c(D,z) &=
 \alpha^2 e^{2\gamma}\prod_{\pi} \left(1-\frac{1}{\rond{N}(\pi)}\right)^2 \prod_{\pi \nmid c}  \left(1-\frac{\rho(\pi)}{\rond{N}(\pi)}\right)^{-1}\\
& \sigma_2(2\tau) \log^2 z \left( 1 +\rond{O}\left(\frac{\tau^5}{\log z}\right)\right),
\end{align*}
where $\tau= \log D / \log z$.
\end{thm}

\begin{proof} The idea is to adapt the Selberg sieve, introducing a free parameter $\lambda$ of support $D$. Due to the hypothesis \eqref{sieve:eq:5} one obtains two quadratic forms $Q_1$ and $Q_2$ instead of one. By the usual method, one can minimize the first $Q_1$, and report the value of $\lambda$ in $Q_2$. For details, see \cite{fou-mic:chgmt-signe} and \cite{lou:these}. This proves the first part of Theorem\,\ref{sieve:thm}. The second part is a classical result in sieve theory. We followed the exposition made in \cite{hal-ric:sieve}: the first step is to estimate the integral 

\begin{equation*}
T_c(D)=\int_1^D G_c(T) \frac{dt}{t}.
\end{equation*}
We find (compare with \cite{hal-ric:sieve}, p. 149-151)

\begin{align*}
T_c(D) &= \frac{\alpha^2}{6} \prod_\pi \left(1-\frac{1}{\rond{N}(\pi)}\right)^2 \prod_{\pi\nmid c} \left(1-\frac{\rho(\pi)}{\rond{N}(\pi)}\right)^{-1} \log^3 D\\ 
&+\rond{O}\left(\log^2 D\right).
\end{align*}
This estimate is then used to obtain (see \cite{hal-ric:sieve}, p.199-201)

\begin{gather*}
G_c(D,z) =
\alpha^2 e^{2\gamma}\prod_{\pi} \left(1-\frac{1}{\rond{N}(\pi)}\right)^2 \prod_{\pi \nmid c} \left(1-\frac{\rho(\pi)}{\rond{N}(\pi)}\right)^{-1}\\
\quad \sigma_2(2\tau) \log^2 z \left( 1 +\rond{O}\left(\frac{\tau^5}{\log z}\right)\right).
\end{gather*}
\end{proof}

Now, it remains to apply Theorem\,\ref{sieve:thm} to the sequence $\rond{A}^\pm$. 
With our definition of $\rho$ and $F$, the estimation of $G_c(D,z)$ simplifies to 

\begin{align}
G_c(D,z)&= \alpha^2 e^{2\gamma}\prod_{\pi} \left(1-\frac{1}{\rond{N}(\pi)}\right)^2 \prod_{\substack{\pi \nmid c\\ \pi\neq \lambda}} \left(1-\frac{2}{\rond{N}(\pi)}\right)^{-1}\nonumber\\
& \sigma_2(2\tau) \log^2 z \left( 1 +\rond{O}\left(\frac{\tau^5}{\log z}\right)\right)\nonumber\\
&\begin{aligned}
&=\alpha^2  e^{2\gamma} \prod_{\pi\tq c\lambda} \left(1-\frac{1}{\rond{N}(\pi)}\right)^2 \prod_{\pi\nmid c\lambda}  \left(1-\frac{1}{\rond{N}(\pi)}\right)^2\left(1-\frac{2}{\rond{N}(\pi)}\right)^{-1} \\
& \sigma_2(2\tau) \log^2 z \left( 1 +\rond{O}\left(\frac{\tau^5}{\log z}\right)\right).
\end{aligned}
\label{sieve:eq:9}
\end{align}
In particular, for a prime $\pi\neq \lambda$ such that $\rond{N}(\pi)\leqslant D$, the estimation $\tau -\log \rond{N}(\pi)/\log z=\rond{O}(\tau)$ is valid, and we deduce from \eqref{sieve:eq:9} that

\begin{equation}\label{sieve:eq:10}
\begin{split}
G_\pi(D/\rond{N}(\pi),z)
&= \alpha^2 \frac{4}{9} e^{2\gamma} F_\lambda(1) \left(1-\frac{2}{\rond{N}(\pi)}\right)\\
& \sigma_2\left( 2\tau -2 \frac{\log \rond{N}(\pi)}{\log z}\right) \log^2 z \left( 1 +\rond{O}\left(\frac{\tau^5}{\log z}\right)\right).
\end{split}
\end{equation}
It follows also from \eqref{sieve:eq:9} that

\begin{equation}\label{sieve:eq:11}
G(D,z)=\alpha^2 \frac{4}{9} e^{2\gamma} F_\lambda(1) \sigma_2\left( 2\tau\right) \log^2 z \left( 1 +\rond{O}\left(\frac{\tau^5}{\log z}\right)\right).
\end{equation}
Now, combining the estimates \eqref{sieve:eq:10} and \eqref{sieve:eq:11}, we obtain from Theorem\,\ref{sieve:thm} that

\begin{equation}\label{sieve:eq:12}
\begin{split}
S(\rond{A}^\pm,z) &\leqslant Y G(D,z)^{-1} +
\frac{2Z}{G(D,z)} \left(1+\rond{O}\left(\frac{\tau^5}{\log z}\right)\right)\\
& \frac{1}{\sigma_2(2\tau)}\sum_{\substack{\pi\tq \rond{P}(z)\\ \pi\neq \lambda}} \frac{\log \rond{N}(\pi)}{\rond{N}(\pi)-2} \sigma_2\left(2\tau -2\frac{\log \rond{N}(\pi)}{\log z}\right) \\
& + \sum_{\substack{d\tq \rond{P}(z) \\ \rond{N}(d) \leqslant D^2}} 3^{\Omega(d)} \tq r(d)\tq.
\end{split}
\end{equation}
Then, by \eqref{sieve:eq:8} and a partial integration, one shows that 

\begin{equation*}
\frac{1}{\sigma_2(2\tau)}\sum_{\substack{\pi\tq \rond{P}(z)\\ \pi\neq \lambda}} \frac{\log \rond{N}(\pi)}{\rond{N}(\pi)-2} \sigma_2\left(2\tau -2\frac{\log \rond{N}(\pi)}{\log z}\right)=\log z+\rond{O}(1).
\end{equation*}
By inserting this inequality in \eqref{sieve:eq:12}, replacing $Y,Z$ by their values given in \eqref{sieve:eq:3} and using the estimation \eqref{sieve:eq:11} of $G(D,z)$, we obtain

\begin{equation}\label{sieve:eq:13}
\begin{split}
S(\rond{A}^\pm ,z) \leqslant &
 \hat{g}(1) e^{-2\gamma} \frac{X}{\sigma_2(2\tau)} \left( \frac{\log X}{\log^2 z}+ \frac{2}{\log z}\right)\\
& + \rond{O}\left( \frac{X\tau^5}{\log^2 z}\right) + \sum_{\substack{d\tq \rond{P}(z) \\ \rond{N}(d) \leqslant D^2}} 3^{\Omega(d)} \tq r(d)\tq.
\end{split}
\end{equation}
We can now choose $z$ and $D$ in terms of $u$ and $X$. Firstly, we set $z=X^{1/u}$. On the one side, the parameter $D$  has to be as large as possible, to optimize the decreasing function $h_2$ in Proposition\,\ref{sieve:prop:2}. On the other side, the last sum in \eqref{sieve:eq:13} has to behave as a remainder term. We set $D=X^{1/4}\log^{-40}X$. With this choice, one has $\tau= u/4 - 40 u \log\log X / \log X$, and therefore 

\begin{equation*}
\frac{1}{\sigma_2(2\tau)} = \frac{1}{\sigma_2(u/2)} \left(1+\rond{O}\left(\frac{1}{\log\log X}\right)\right).
\end{equation*}
Thus \eqref{sieve:eq:13} yields

\begin{equation}\label{sieve:eq:14}
\begin{split}
S(\rond{A}^\pm ,z) \leqslant &
\hat{g}(1) e^{-2\gamma} \frac{X}{\sigma_2(u/2) \log X} \left(u^2+ 2u\right)\\
& +  \rond{O}\left( \frac{X}{\log X \log\log X}\right) + \sum_{\substack{d\tq \rond{P}(z) \\ \rond{N}(d) \leqslant D^2}} 3^{\Omega(d)} \tq r(d)\tq.
\end{split}
\end{equation}
We need the following  property of $\sigma_2$:

\begin{equation*}
\frac{1}{\sigma_2(t)} =1 -h_2(t), \qquad \ec{with } h_2(t)=\rond{O}\left( \left(\frac{t}{2}\right)^{-t/2}\right), \qquad \ec{for } t >1.
\end{equation*}
This gives finally

\begin{equation}\label{sieve:eq:15}
\begin{split}
S(\rond{A}^\pm ,z) &\leqslant \hat{g}(1) e^{-2\gamma} \frac{X}{\log X} (u^2+2u) \left(1-h_2(u/2)\right) \\
& +  \rond{O}\left( \frac{X}{\log X \log\log X}\right) + \sum_{\substack{d\tq \rond{P}(z) \\ \rond{N}(d) \leqslant D^2}} 3^{\Omega(d)} \tq r(d)\tq.
\end{split}
\end{equation}
The last sum in \eqref{sieve:eq:15}, i.e. the contribution of the cubic sums, is evaluated by the Cauchy-Schwarz inequality and the Weil bound \eqref{Kloo:eq:Weil}. We obtain 

\begin{align*}
&\Bigg\vert \sum_{\rond{N}(d) \leqslant D^2}3^{\Omega(d)}\  r(d)\Bigg\vert  \leqslant \\
&\Bigg\vert \sum_{\substack{\lambda \nmid d \\ \rond{N}(d)\leqslant D^2}} 9^{\Omega(d)} \sum_{\substack{\lambda \nmid d\\ c\congru 0\pmod d}} \frac{\pm S(1,c)}{\sqrt{\rond{N}(c)}}g\left(\frac{\rond{N}(c)}{X}\right)\Bigg\vert^{1/2}\\
&\Bigg\vert \sum_{\substack{\lambda \nmid d \\ \rond{N}(d)\leqslant D^2}}  \sum_{\substack{\lambda \nmid d\\ c\congru 0\pmod d}} \frac{\pm S(1,c)}{\sqrt{\rond{N}(c)}}g\left(\frac{\rond{N}(c)}{X}\right)\Bigg\vert^{1/2}\\
& \leqslant X^{1/2} \log^{10}X \sqrt{ \Sigma(D^2)} +\rond{O}\left(X^{1/2-\varepsilon}\right).
\end{align*}
Our choice of $D=X^{1/4}\log^{-40}X$ yields, via Corollary\,\ref{intro:cor:2}, the desired bound for the error term. By inserting it in \eqref{sieve:eq:15}, we conclude the proof of the first statement of Proposition\,\ref{sieve:prop:2}. For the estimation of $h_2(u)$, we refer to \cite{siv:kloo}, Lemme\,3.4, where precise estimates are given. Thereby, we conclude the proof of \eqref{intro:eq:upper}, with the function $h(u)$ in \eqref{intro:eq:upper} given by $e^{-2\gamma} (u^2+2u)(\tq h_1(u)\tq +\tq h_2(u)\tq)$. In follows from the estimates on $h_1(u)$ and $h_2(u)$ given in Proposition\,\ref{sieve:prop:1} and Proposition\,\ref{sieve:prop:2} that, for $u=60$, one has $ h(u)< e^{-2\gamma} (u^2+2u) 10^{-8}$; in particular, a numerical computation shows that $h(60) <10^{-4}$.  

\section{Vertical Sato-Tate law}\label{sec:sato-tate}

In this section  we prove the lower bound \eqref{intro:eq:lower}. The idea is to show that when the moduli have a limited number of prime factors, then sums of absolute values of cubic sums are not too small. 

Let us start by considering Eisenstein integers $c$ of the form $c=\pi_1 \pi_2 \pi_3$, where the $\pi_i$'s are prime integers of $R$. Then, by the twisted multiplicativity \eqref{Kloo:eq:mult},

\begin{align*}
\frac{S(1,c)}{\sqrt{\rond{N}(c)}} &= 
\frac{S(\barre{\pi_2\pi_3},\pi_1)}{\sqrt{\rond{N}(\pi_1)}} \frac{S(\barre{\pi_1\pi_3},\pi_2)}{\sqrt{\rond{N}(\pi_2)}} \frac{S(\barre{\pi_1\pi_2},\pi_3)}{\sqrt{\rond{N}(\pi_3)}}\\
&= 8 \cos \theta_{\pi_1,\barre{\pi_2\pi_3}} \cos \theta_{\pi_2,\barre{\pi_1\pi_3}} \cos \theta_{\pi_3,\barre{\pi_1\pi_2}},
\end{align*}
where the angles $\theta_{\pi,m}$ are defined by 

\begin{equation*}
\sum_{x\pmod \pi} e\left(\frac{m(x^3-3x)}{\pi}\right) = 2 \cos \theta_{\pi,m} \rond{N}(\pi)^{1/2}, \quad \pi \nmid m.
\end{equation*}
In order to simplify the notations, let us write $\theta_1=\theta_{\pi_1,\barre{\pi_2\pi_3}}$ and define similarly $\theta_2$ and $\theta_3$. Let us choose some numbers $\mu_3^\pm, \mu_2^\pm$ such that 

\begin{align*}
&\mu_3^-<\mu_3^+<\mu_2^-<\mu_2+\\
& \mu_2^+ < 1-(\mu_2^++\mu_3^+) < 1- (\mu_2^-+\mu_2^-)<\mu_2^-+\mu_2^+.
\end{align*}
More precisely, we take $\mu_3^-=7/2$, $\mu_3^+=\mu_2^-=13/42$ and $\mu_2^+=1/3$. Define $P_3^-=X^{\mu_3^-}$, $P_2^-=X^{\mu_2^-}$ and

\begin{align*}
P_3^+=\max_{\lambda=1,2,\pts}\left\{2^\lambda X^{\mu_3^-} : 2^\lambda X^{\mu_3^-}<\frac{X^{\mu_3^+}}{2}\right\}\\
P_2^+=\max_{\lambda=1,2,\pts}\left\{2^\lambda X^{\mu_2^-} : 2^\lambda X^{\mu_2^-}<\frac{X^{\mu_2^+}}{2}\right\}.
\end{align*}

In order that the angles $\theta_i$, $i=1,2,3$, are well defined, we impose the conditions $X^{\mu_3^-}\leqslant \rond{N}(\pi_3)< X^{\mu_3^+}$ and $X^{\mu_2^-}\leqslant \rond{N}(\pi_2)< X^{\mu_2^+}$.
Moreover, let us choose a parameter $t\in[0,\pi/2]$ such that the interval $I=[0,t] \cup [\pi-t,\pi]$ of $[0,\pi]$ has measure $\mu_{ST}(I)=3/4$. 
Let $g$ be a smooth function with compact support included in $[1,2]$. Then, one has

\begin{align}
&\sum_{\pi\tq c \Rightarrow \rond{N}(\pi)>X^{1/u}} \left\vert \frac{S(1,c)}{\sqrt{\rond{N}(c)}}\right\vert g\left( \frac{\rond{N}(c)}{X}\right) \nonumber \\
&\geqslant \sum_{\substack{\rond{N}(\pi_3)> X^{1/u}\\P_3^-\le \rond{N} \pi_3< 2P_3^+}} \sum_{\substack{\rond{N}(\pi_2)> X^{1/u}\\P_2^-\le \rond{N} \pi_2< 2P_2^+}} 
\sum_{\substack{\rond{N}(\pi_1)>X^{1/u}\\ \theta_1, \theta_2, \theta_3\in I}} \left\vert \frac{S(1,\pi_1\pi_2\pi_3)}{\sqrt{\rond{N}(\pi_1\pi_2\pi_3)}}\right\vert g\left( \frac{\rond{N}(\pi_1\pi_2\pi_3)}{X}\right) \nonumber \\
&\geqslant 8 (\cos t)^3
\sum_{P_2^-\le \rond{N} \pi_2< 2P_2^+} \sum_{P_3^-\le \rond{N} \pi_3< 2P_3^+}  \sum_{\substack{\pi_1\\\theta_1,\theta_2,\theta_3\in I}} g\left(\frac{\rond{N}(\pi_1\pi_2\pi_3)}{X}\right).\label{sato-tate:eq:1}
\end{align}
Let us introduce the notations 

\begin{align*}
m(E_i) &= 
\sum_{P_2^-\le \rond{N} \pi_2< 2P_2^+} \sum_{P_3^-\le \rond{N} \pi_3< 2P_3^+}  \sum_{\substack{\pi_1\\\theta_i\in I}} g\left(\frac{\rond{N}(\pi_1\pi_2\pi_3)}{X}\right),\\
m(E) &= 
\sum_{P_2^-\le \rond{N} \pi_2< 2P_2^+} \sum_{P_3^-\le \rond{N} \pi_3< 2P_3^+}  \sum_{\pi_1} g\left(\frac{\rond{N}(\pi_1\pi_2\pi_3)}{X}\right).
\end{align*}
Then a simple inclusion-exclusion argument shows that

\begin{equation}\label{sato-tate:eq:2}
\sum_{P_2^-\le \rond{N} \pi_2< 2P_2^+} \sum_{P_3^-\le \rond{N} \pi_3< 2P_3^+}  \sum_{\substack{\pi_1\\\theta_1,\theta_2,\theta_3\in I}} g\left(\frac{\rond{N}(\pi_1\pi_2\pi_3)}{X}\right)\\
\geqslant \left( m(E_1) + m(E_2) +m(E_3) -2 m(E)\right).
\end{equation}

Now, the Sato-Tate vertical law states that the angles $\theta_{\pi,a}$ are equidistributed, when $a$ runs modulo $\pi$. Therefore each $m(E_i)$ should be independent of $i$, and depending only on the size of the interval $I$. This is actually true, and is expressed in the next proposition:

\begin{prop}\label{sato-tate:prop} With the above notations,

\begin{equation*}
m(E_i) =(\mu_{ST}(I)+o(1)) m(E) + \rond{O}\left(\frac{X}{\log^2 X}\right).
\end{equation*}
\end{prop} 

\begin{proof} 
The analog statement has been established for Kloosterman sums in \cite[Lemme\,5.1]{fou-mic:chgmt-signe}. The proof is based on two main results. First, one shows, as in \cite[Lemme\,2.3]{fou-mic:sommes-mod} that the vertical Sato-Tate law (Theorem\,\ref{intro:thm:vert}) implies that 

\begin{equation}\label{sato-tate:prop:eq1}
\frac{1}{\rond{N}\pi -1} \sum_{\substack{0\neq a\pmod \pi\\ \theta_{\pi,\barre{a}}\in I }} 1 = \mu_{ST}(I) + \rond{O}((\rond{N}\pi)^{-1/8}).
\end{equation}
Secondly, one uses the following large sieve inequality

\begin{equation}\label{sato-tate:prop:eq2}
\begin{split}
\sum_{P<\rond{N}\pi\le 2P} \ \sum_{0\neq a\pmod \pi} \Bigg\vert \sum_{\substack{\rond{N} n\le X\\ n\congru a \pmod \pi}} f(n) g\left(\frac{\rond{N} (\pi n)}{Y}\right) - \frac{1}{\rond{N}\pi -1} \sum_{(n,\pi)=1} f(n) f\left(\frac{\rond{N} (np)}{Y}\right)\Bigg\vert^2 \\
\ll \left( \frac{X}{P}+P\right) \left( \sum_{\rond{N} n\le X} \tq f(n)\tq^2\right),
\end{split}
\end{equation}
for any arithmetic function $f$. The proof of \eqref{sato-tate:prop:eq2} follows readily \cite[Lemme\,2.5]{fou-mic:sommes-mod}, using the orthogonality relation and the large sieve inequality for multiplicative characters over a number field (see for example \cite[Theorem\,4]{hux}). Proposition\,\ref{sato-tate:prop} then follows from \eqref{sato-tate:prop:eq1} and \eqref{sato-tate:prop:eq2}, as in \cite[Proposition\,4.1]{fou-mic:sommes-mod}.
\end{proof}

Using \eqref{sato-tate:eq:1}, \eqref{sato-tate:eq:2} and Proposition\,\ref{sato-tate:prop}, we get the lower bound

\begin{align*}
&\sum_{\pi\tq c \Rightarrow \rond{N}(\pi)>X^{1/u}} \left\vert \frac{S(1,c)}{\sqrt{\rond{N}(c)}} g\left( \frac{\rond{N}(c)}{X}\right) \right\vert \nonumber \\
&\geqslant 8 \cos^3 t (3\mu_{ST}(I)-2) \mu(E)+ \rond{O}\left( \frac{X}{\log^2 X}\right).
\end{align*}
With our choice of the interval $I$, we have $\mu_{ST}(I)=3/4$ and $\cos^3 t> 0.0075$. Finally, the prime number theorem gives

\begin{equation*}
\sum_{\pi\tq c \Rightarrow \rond{N}(\pi)>X^{1/u}} \left\vert \frac{S(1,c)}{\sqrt{\rond{N}(c)}} g\left( \frac{\rond{N}(c)}{X}\right) \right\vert 
\geqslant \frac{3}{2} 10^{-2} \hat{g}(1) \frac{X}{\log X}.
\end{equation*}
This proves the inequality \eqref{intro:eq:lower} stated in the introduction.


\end{document}